\documentclass[11pt,english]{amsart}
\usepackage[T1]{fontenc}
\usepackage[latin1]{inputenc}
\usepackage{verbatim}
\usepackage{textcomp}
\usepackage{amstext}
\usepackage{amsthm}
\usepackage{amssymb}

\makeatletter

\providecommand{\tabularnewline}{\\}

\numberwithin{equation}{section}
\numberwithin{figure}{section}
\theoremstyle{plain}
\newtheorem{thm}{\protect\theoremname}
\theoremstyle{definition}
\newtheorem{defn}[thm]{\protect\definitionname}
\theoremstyle{definition}
\newtheorem{problem}[thm]{\protect\problemname}
\theoremstyle{remark}
\newtheorem{rem}[thm]{\protect\remarkname}
\theoremstyle{plain}
\newtheorem*{thm*}{\protect\theoremname}
\theoremstyle{definition}
\newtheorem{example}[thm]{\protect\examplename}
\theoremstyle{plain}
\newtheorem{prop}[thm]{\protect\propositionname}
\theoremstyle{plain}
\newtheorem{lem}[thm]{\protect\lemmaname}
\newenvironment{lyxlist}[1]
	{\begin{list}{}
		{\settowidth{\labelwidth}{#1}
		 \setlength{\leftmargin}{\labelwidth}
		 \addtolength{\leftmargin}{\labelsep}
		 }}
	{\end{list}}

\usepackage{tikz}
\usepackage{pgfplots}

\makeatother

\makeatother

\usepackage{babel}
\providecommand{\definitionname}{Definition}
\providecommand{\examplename}{Example}
\providecommand{\lemmaname}{Lemma}
\providecommand{\problemname}{Problem}
\providecommand{\propositionname}{Proposition}
\providecommand{\remarkname}{Remark}
\providecommand{\theoremname}{Theorem}

\begin{document}

\title[On some operators acting on line arrangements]{On some operators acting on line arrangements and their dynamics}

\addtolength{\textwidth}{0mm}
\addtolength{\hoffset}{-0mm} 
\addtolength{\textheight}{0mm}
\addtolength{\voffset}{-0mm} 


\global\long\def\CC{\mathbb{C}}%
 
\global\long\def\BB{\mathbb{B}}%
 
\global\long\def\PP{\mathbb{P}}%
 
\global\long\def\QQ{\mathbb{Q}}%
 
\global\long\def\RR{\mathbb{R}}%
 
\global\long\def\FF{\mathbb{F}}%

\global\long\def\DD{\mathbb{D}}%
 
\global\long\def\NN{\mathbb{N}}%
\global\long\def\ZZ{\mathbb{Z}}%
 
\global\long\def\HH{\mathbb{H}}%
 
\global\long\def\Gal{{\rm Gal}}%

\global\long\def\bA{\mathbf{A}}%

\global\long\def\kP{\mathfrak{P}}%
 
\global\long\def\kQ{\mathfrak{q}}%
 
\global\long\def\ka{\mathfrak{a}}%
\global\long\def\kP{\mathfrak{p}}%
\global\long\def\kn{\mathfrak{n}}%
\global\long\def\km{\mathfrak{m}}%

\global\long\def\cA{\mathfrak{\mathcal{A}}}%
\global\long\def\cB{\mathfrak{\mathcal{B}}}%
\global\long\def\cC{\mathfrak{\mathcal{C}}}%
\global\long\def\cD{\mathcal{D}}%
\global\long\def\cH{\mathcal{H}}%
\global\long\def\cK{\mathcal{K}}%

\global\long\def\cF{\mathcal{F}}%
 
\global\long\def\cI{\mathfrak{\mathcal{I}}}%
\global\long\def\cJ{\mathcal{J}}%

\global\long\def\cL{\mathcal{L}}%
\global\long\def\cM{\mathcal{M}}%
\global\long\def\cN{\mathcal{N}}%
\global\long\def\cO{\mathcal{O}}%
\global\long\def\cP{\mathcal{P}}%
\global\long\def\cS{\mathcal{S}}%
\global\long\def\cW{\mathcal{W}}%

\global\long\def\cQ{\mathcal{Q}}%
\global\long\def\kBS{\mathfrak{B}_{6}}%

\global\long\def\a{\alpha}%
 
\global\long\def\b{\beta}%
 
\global\long\def\d{\delta}%
 
\global\long\def\D{\Delta}%
 
\global\long\def\L{\Lambda}%
 
\global\long\def\g{\gamma}%
\global\long\def\om{\omega}%

\global\long\def\G{\Gamma}%
 
\global\long\def\d{\delta}%
 
\global\long\def\D{\Delta}%
 
\global\long\def\e{\varepsilon}%
 
\global\long\def\k{\kappa}%
 
\global\long\def\l{\lambda}%
 
\global\long\def\m{\mu}%

\global\long\def\o{\omega}%
 
\global\long\def\p{\pi}%
 
\global\long\def\P{\Pi}%
 
\global\long\def\s{\sigma}%

\global\long\def\S{\Sigma}%
 
\global\long\def\t{\theta}%
 
\global\long\def\T{\Theta}%
 
\global\long\def\f{\varphi}%
 
\global\long\def\ze{\zeta}%

\global\long\def\deg{{\rm deg}}%
 
\global\long\def\det{{\rm det}}%

\global\long\def\Dem{Proof: }%
 
\global\long\def\ker{{\rm Ker}}%
 
\global\long\def\im{{\rm Im}}%
 
\global\long\def\rk{{\rm rk}}%
 
\global\long\def\car{{\rm car}}%
\global\long\def\fix{{\rm Fix( }}%

\global\long\def\card{{\rm Card }}%
 
\global\long\def\codim{{\rm codim}}%
 
\global\long\def\coker{{\rm Coker}}%

\global\long\def\pgcd{{\rm pgcd}}%
 
\global\long\def\ppcm{{\rm ppcm}}%
 
\global\long\def\la{\langle}%
 
\global\long\def\ra{\rangle}%

\global\long\def\Alb{{\rm Alb}}%
 
\global\long\def\Jac{{\rm Jac}}%
 
\global\long\def\Disc{{\rm Disc}}%
 
\global\long\def\Tr{{\rm Tr}}%
 
\global\long\def\Nr{{\rm Nr}}%

\global\long\def\NS{{\rm NS}}%
 
\global\long\def\Pic{{\rm Pic}}%

\global\long\def\Km{{\rm Km}}%
\global\long\def\rk{{\rm rk}}%
\global\long\def\Hom{{\rm Hom}}%
 
\global\long\def\End{{\rm End}}%
 
\global\long\def\aut{{\rm Aut}}%
 
\global\long\def\SSm{{\rm S}}%

\global\long\def\psl{{\rm PSL}}%
 
\global\long\def\cu{{\rm (-2)}}%
 
\global\long\def\mod{{\rm \,mod\,}}%
 
\global\long\def\cros{{\rm Cross}}%
 
\global\long\def\nt{z_{o}}%

\global\long\def\co{\mathfrak{\mathcal{C}}_{0}}%
\global\long\def\ci{\mathfrak{\mathcal{C}}_{1}}%
\global\long\def\ldt{\Lambda_{\{2\},\{3\}}}%
 
\global\long\def\ltd{\Lambda_{\{3\},\{2\}}}%

\author{Xavier Roulleau}
\begin{abstract}
We study some natural operators acting on configurations of points
and lines in the plane and remark that many interesting configurations
are fixed points for these operators. We review ancient and recent
results on line or point arrangements though the realm of these operators.
We study the first dynamical properties of the iteration of these
operators on some line arrangements. 
\end{abstract}

\maketitle
\tableofcontents{}

\section{Introduction}

Line arrangements is an active field of research which is attractive
from many view points: Topology (Zariski pairs), Algebra (freeness),
Combinatorics... For algebraic geometers, the results of Hirzebruch
in \cite{Hirz} have been the starting point for the search of peculiar
line arrangements in order to geometrically construct some ball quotient
surfaces, which are surfaces with Chern numbers satisfying the equality
in the Bogomolov-Miyaoka-Yau inequality $c_{1}^{2}\leq3c_{2}$. Arrangements
of lines appear also to be important for the Bounded Negativity Conjecture
on algebraic surfaces \cite{Duke},\,\cite{IMRN}. 

Motivated by the problem of constructing new interesting line arrangements,
we formalize here the definitions of some operators acting on lines
and points arrangements and we ask natural questions about their properties.
For integers $n,m$ $\geq2$, one may for example define the line
operator $\L_{n,m}$ as the function which to a given line arrangement
$L_{0}$ associate the union of the lines that contain at least $m$
points among the points in $L_{0}$ that have multiplicity at least
$n$. The idea of that definition comes naturally when considering
the definition of $(r_{k},s_{m})$-configurations (see Section \ref{subsec:The-(r,n)-configurations}).
These operators are probably known and used, however, to the knowledge
of the author, they are not studied for themselves and this is one
of our aims to (start to) do so and to convince the reader of their
importance. 

We will show that these operators appear quite naturally and ubiquitously
when studying line arrangements with remarkable properties, such as
reflexion line arrangements, simplicial arrangements, Zariski pairs
or Sylvester-Gallai line arrangements.

For example, a Zariski pair is a pair of plane curves $\cA,\cA'$
which have the same combinatoric, but have different topological properties,
which properties are often difficult to detect. We obtain that the
line operators are able to detect geometric differences between Zariski
pairs of line arrangements $\cA$ and $\cA'$ when the moduli space
of such $\cA,\cA'$ has dimension $\geq1$, but they do not see any
differences when their moduli space is $0$ dimensional. 

Some natural questions on these operators $\L$ arise, for example
the question to understand what are their fixed points i.e. if we
can classify the arrangements $\cA$ such that $\L(\cA)=\cA$. As
we explain in this paper, the answer is known for the $\L_{2,2}$-operator,
as a result of De Bruijn--Erdös. We give some examples of fixed points
of operators $\L$, but the question to classify them is quite open.
It is also an interesting question to construct line arrangements
$\cA$ such that $\cA$ and $\L(\cA)$ have the same combinatorics
but with $\L(\cA)\neq\cA$, so that $\L$ is a (rational) self-map
on the moduli space of these arrangements; we study such lines arrangements
on some cases. We also show that many results on line arrangements
may be formulated as results on these operators $\L$. 

Another natural question is to understand line arrangements for which
the action of an operator $\L$ diverges, i.e. for which the successive
images $(\cA_{k})_{k\geq0}$ by $\L$ of an arrangement $\cA_{0}$
acquire more and more lines. 

The other operators on line arrangements the author is aware of are
related to the pentagram map defined by Schwartz in \cite{SchwartzPent}.
To a polygon of $m$ lines i.e. a $m$-tuple of lines $(\ell_{1},\dots,\ell_{m})$,
the pentagram map associates the polygon $(\ell_{1}',\dots,\ell_{m}')$,
where $\ell_{k}'$ is the line containing the points $p_{k},p_{k+2}$,
for $p_{k}=\ell_{k}\cap\ell_{k+1}$ (the indices being taken mod $m$).
That creates a rational self-map on the space of polygons, which is
of interest in classical projective geometry, algebraic combinatorics,
moduli spaces, cluster algebras and integrable systems, see e.g. \cite{OST},
\cite{Berger}.

The structure of the paper is as follows: 

In Section \ref{sec:Definitions-of-the}, we define the line and point
operators and give some speculations on their properties and the dynamics
of their actions. The ring generated by these operators acts naturally
on the free module of line (respectively point) arrangements. We then
remark that the Hirzebruch type inequalities could be understood by
using these operators and factorizing through the Chow group of points
of the plane. We also review some classical theorems (Pappus, Pascal,
Desargues) and remark that they can be unified and formulated as statements
on the operator $\L_{\{2\},\{3\}}.$ 

In Section \ref{sec:Some-examples}, we study some examples of action
of operators on classical arrangements of lines (such as reflexion
arrangements), on simplicial arrangements, on some free arrangements,
and Zariski pairs. We also obtain an example of an arrangement of
81 lines with H-constant $<-3$, an interesting property for the Bounded
Negativity Conjecture. Although we do not claim for completeness,
these two first Sections may be also of interest for readers willing
to have a short survey on line arrangements and related problems.

The next Sections contain more original results: 

In Section \ref{sec:First-divergent-arrangements}, we explore the
divergence of some operators such as $\L_{2,3},$ $\L_{\{2\},\{3\}},$
$\L_{3,2}$ and $\L_{3,3}$ on some line arrangements. 

Section \ref{Sec:On-the-L=00007B2=00007D=00007B3=00007D-operator}
is devoted to the line operator $\ldt$. We discuss about two families
of arrangements of six lines, which we call the unassuming arrangements
and the flashing arrangements (we study in details the former in \cite{RoulleauDy}),
for which the action of $\L_{\{2\},\{3\}}$ is remarkable, recalling
some classical theorems of geometry. 

In Section \ref{sec:Flashing-arrangements-for}, we construct flashing
arrangements i.e. line arrangements $\cA$ such that $\L(\cA)\neq\cA$
but $\L(\L(\cA))=\cA$ for operators $\L\in\{\L_{\{2\},\{k\}},\,|\,k=3,4,5,6\}$. 

We conclude in Section \ref{sec:Matroids-and-line} by recalling some
notions used in this paper about matroids and by some remarks on realizable
rank $3$ matroids and the line operators $\L$. 

The computations involved use the Magma algebra software \cite{Magma};
one can find our algorithms in the ancillary file of the arXiv version
of this paper. 

\textbf{Acknowledgements} The author thanks Lukas Kühne for his explanations
on the realization of matroids and pointing out the paper \cite{ACKN},
Benoît Guerville-Ballé for instructive discussions on Zariski pairs,
and Jérémy Blanc, Alberto Calabri and Ciro Ciliberto for interesting
discussions related the Cremona group and arrangements in Section
\ref{subsec:Povera-arrangements}. The author is also grateful to
Piotr Pokora for his helpful comments on this paper, and his useful
observations about configurations and $2$-arrangements in Section
\ref{subsec:-arrangements}. Support from Centre Henri Lebesgue ANR-11-LABX-0020-01
is acknowledged.

\section{\label{sec:Definitions-of-the}Definitions and occurrences of the
operators, first questions}

\subsection{Notations and usual definitions}

Let $K$ be a field. We denote by $\PP^{2}$ (respectively $\check{\PP}^{2}$)
the set of points defined over $K$ of the projective plane (respectively
of the dual projective plane). 

For a set $E$, we denote by $\mathcal{P}(E)^{\text{F}}$ the set
of finite subsets of $E$. An element of $\mathcal{P}(\check{\PP}^{\text{2}})^{\text{F}}$
is called a \textit{line arrangement}, and an element of $\mathcal{P}(\PP^{\text{2}})^{\text{F}}$,
a \textit{point configuration}. An arrangement of lines will be identified
with the union of these lines in $\PP^{2}$. 

Let $L_{0}$ be an arrangement of lines. For $k\geq2$, a point $p$
is said \textit{multiple} with multiplicity $k$ for $L_{0}$, (or,
for short, a $k$\textit{-point} of $L_{0}$) if $p$ is at the intersection
of exactly $k$ lines in $L_{0}$. We denote by $t_{k}(L_{0})$ the
number of $k$-points of $L_{0}$.

Let $P_{0}$ be a point configuration. A line which contains exactly
$k$ points in $P_{0}$ is said $k$\textit{-rich}. 

The cardinal of a set $E$ is denoted by $|E|$.

When a line arrangement is known under the name of his inventor e.g
Wiman, we refer to it as the ``Wiman configuration'' rather than
the ``Wiman arrangement'', as is customary in the literature, and
especially if it is part of a $(r_{k},s_{m})$-configuration as defined
in Section \ref{subsec:The-(r,n)-configurations}. 

\subsection{\label{subsec:The-lines-points-oper}The lines and points operators
$\protect\L$ and $\Psi$}

For a subset $\kn\subset\NN\setminus\{0,1\}$ of integers larger or
equal than $2$, let us define the functions 
\[
\mathcal{L}_{\kn}:\mathcal{P}(\PP^{2})^{\text{F}}\to\mathcal{P}(\check{\PP}^{\text{2}})^{\text{F}}\text{ and }\mathcal{P}_{\km}:\mathcal{P}(\check{\PP}^{\text{2}})^{\text{F}}\to\mathcal{P}(\PP^{2})^{\text{F}}
\]
hereby: 

$\bullet$ To a configuration $P_{0}$ of points in the plane, let
$\mathcal{L}_{\kn}(P_{0})$ be the set of lines defined by
\[
\mathcal{L}_{\kn}(P_{0})=\{\ell\,|\,\,|\ell\cap P_{0}|\in\kn\}
\]
i.e. $\mathcal{L}_{\kn}(P_{0})$ is the set of lines $\ell$ in the
plane such that the number of points of $P_{0}$ contained in $\ell$
is in the set $\mathfrak{n}$.

$\bullet$ To an arrangement of lines $L_{0}$, let $\text{\ensuremath{\mathcal{P}}}_{n}(L_{0})$
be the set of points
\[
\text{\ensuremath{\mathcal{P}}}_{\mathfrak{n}}(L_{0})=\{p\,|\,\exists\,k\in\kn,\,\,p\text{ is a }k\text{-point of }L_{0}\}.
\]

Then for subsets $\kn,\km\subset\NN\setminus\{0,1\}$, let us define
the $(\kn,\km)$-\textit{line operator} and the $(\kn,\km)$-\textit{point
operator} 
\[
\Lambda_{\kn,\km}:\mathcal{P}(\check{\PP}^{\text{2}})^{\text{F}}\to\mathcal{P}(\check{\PP}^{\text{2}})^{\text{F}},\,\,\Psi_{\kn,\km}:\mathcal{P}(\PP^{\text{2}})^{\text{F}}\to\mathcal{P}(\PP^{2})^{\text{F}}
\]
by 
\[
\Lambda_{\kn,\km}(L_{0})=\mathcal{L}_{\km}(\cP_{\kn}(L_{0}))\text{ and }\Psi_{\kn,\km}(P_{0})=\cP_{\km}(\mathcal{L}_{\kn}(P_{0})).
\]
If $\mathfrak{n}=\mathfrak{m}$, we may simply write $\L_{\mathfrak{n}}=\L_{\mathfrak{n},\mathfrak{n}}$
and $\Psi_{\mathfrak{n},\mathfrak{n}}=\Psi_{\mathfrak{n}}$. For $n\in\NN,n\geq2$,
we will use the notation $\cL_{n},\cP_{n}$ for 
\[
\cL_{n}=\cL_{\{k\geq n\,|\,k\in\NN\}},\,\,\cP_{n}=\cP_{\{k\geq n\,|\,k\in\NN\}},
\]
and $\L_{n}=\mathcal{L}_{n}\circ\cP_{n}$, $\Psi_{n}=\cP_{n}\circ\mathcal{L}_{n}.$

For integers $n\geq2$, the operators $\cL_{\{n\}}$, $\cP_{\{n\}}$
are elementary, meaning that 
\[
\cL_{\mathfrak{n}}(P_{0})=\cup_{n\in\mathfrak{n}}\cL_{\{n\}}(P_{0}),\,\,\cP_{\mathfrak{n}}(L_{0})=\cup_{n\in\mathfrak{n}}\cP_{\{n\}}(L_{0}),
\]
for a point configuration $P_{0}$ or a line arrangement $L_{0}$;
moreover these unions are disjoint. We also have the relations 
\begin{equation}
\L_{\mathfrak{m},\mathfrak{n}}(L_{0})=\cup_{(m,n)\in\mathfrak{m}\times\mathfrak{n}}\L_{\{m\},\{n\}}(L_{0}).\label{eq:relations_lambda}
\end{equation}

Let $\L$ be a line operator and $L_{0}$ be a line arrangement.
We define inductively a sequence $(L_{n})_{n\in\NN}$ of lines arrangements,
called the $\L$-sequence associated to $L_{0}$, by the relation
\[
L_{n+1}=\L(L_{n}).
\]

\begin{defn}
We say that the $\L$-sequence $(L_{m})_{m}$ is \textit{$\L$-convergent}
to a line arrangement $L$ if $\exists M\in\NN$ such that $\forall m\geq M$,
$L_{m}=L$. We say that a sequence is \textit{extinguishing} if it
converges to $\emptyset$. If a sequence $(L_{m})_{m}$ is extinguishing,
its \textit{length} is the smallest integer $m$ such that $L_{m}=\emptyset$.
A line arrangement $L$ is $\L$\textit{-fixed} if $\L(L)=L$. The
operator $\L$ is said divergent for $L$ if the number of lines of
the associated $\L$-sequence diverges to infinity. 
\end{defn}

The following questions are then natural:
\begin{problem}
Under which conditions on $(\L,L_{0})$ is the $\L$-sequence $(L_{m})_{m}$
convergent ? Or periodic ? Or extinguishing ?\\
For $n\in\NN$, do there exists a $\L$-sequence of length $n$ ?
Do there exist sequences that are periodic with a non-trivial period
?\\
For a given arrangement $L_{0}$, there always exists a line operator
$\L$ such that $\L(L_{0})=\emptyset$. Do there always exists a operator
$\L$ such that the $\L$-sequence is convergent and non-extinguishing
? Such that $L_{0}$ is $\L$-fixed ?
\end{problem}

We give in Section \ref{sec:Some-examples} some examples which may
help to refine these questions.%

\subsection{\label{subsec:The-dual-operators}The dual operators}

Fixing coordinates $x_{1},x_{2},x_{3}$ of $\PP^{2}$, one can define
the operator $\cD:\cP(\PP^{2})^{\text{F}}\to\cP(\check{\PP^{2}})^{\text{F}}$
sending a point $p=(u_{1}:u_{2}:u_{3})$ to the line (denoted by $^{t}p$)
defined by $u_{1}x_{1}+u_{2}x_{2}+u_{3}x_{3}=0$ and the operator
$\check{\cD}:\cP(\check{\PP^{2}})^{\text{F}}\to\cP(\PP^{2})^{\text{F}}$
sending the line $\ell=\{u_{1}x_{1}+u_{2}x_{2}+u_{3}x_{3}=0\}$ to
the point $[\ell]=(u_{1}:u_{2}:u_{3}).$ 

The arrangement $\cD(P_{0})$ (respectively $\check{\cD}(L_{0})$)
is called the dual of $P_{0}$ (respectively $L_{0}$). The operators
$\cD,\check{\cD}$ are inverse of each others: $\check{\cD}\cD=\text{Id}$,
$\cD\check{\cD}=\text{Id}.$ %

In practice we will forget the accent and write $\cD$ for both $\cD$
and $\check{\cD}$. Also the line and point operators act naturally
on both spaces $\cP(\check{\PP^{2}})^{\text{F}}$ and $\cP(\PP^{2})^{\text{F}}$
and we will take the same notations. By the properties of duality,
one has 
\[
\cD\circ\cP_{\mathfrak{n}}=\cL_{\mathfrak{n}}\circ\cD,\,\,\cL_{\mathfrak{n}}\circ\cD=\cD\circ\cP_{\mathfrak{n}},
\]
so that 
\[
\cD\circ\L_{\mathfrak{m},\mathfrak{n}}\circ\cD=\Psi_{\mathfrak{m},\mathfrak{n}},\,\,\cD\circ\Psi_{\mathfrak{m},\mathfrak{n}}\circ\cD=\L_{\mathfrak{m},\mathfrak{n}}.
\]
For a subset $\mathfrak{n}$ of integers $\geq2$, we denote by $\cD_{\mathfrak{n}}$
the operator 
\[
\cD_{\mathfrak{n}}=\cL_{\mathfrak{n}}\circ\cD,
\]
which to line arrangement $\cA$ associates the line arrangement in
the dual plane which is union of the lines containing exactly $n$
points in $\cD(\cA)$ for $n\in\mathfrak{n}$. It has the advantage
to send line arrangements to line arrangements and that for any two
subsets $\mathfrak{n},\mathfrak{m}$ of integers $\geq2$, the following
relation holds
\[
\cD_{\mathfrak{n}}\circ\cD_{\mathfrak{m}}=\L_{\mathfrak{m},\mathfrak{n}}.
\]

\subsection{\label{subsec:The-ring-operators}The ring of operators on the free
module of line arrangements}

Let $\text{Fr}_{0}(\check{\PP}^{2})=\ZZ{}^{(\cP(\check{\PP}^{2})^{\text{F}})}$
be the free module of line arrangements: an element of $\text{Fr}_{0}$
is a formal finite sum 
\[
\sum a_{L}[L]
\]
where $L$ runs over the arrangements of lines and $a_{L}$ is an
integer which is $0$ unless for a finite number of arrangements.
One can extend linearly the operators on $\text{Fr}_{0}(\check{\PP}^{2})$,
then adding the identity operator, one get the (non-commutative) ring
of operators $\ZZ[\L]$. 
\begin{problem}
Are there relations among the line operators in the ring $\ZZ[\L]$
?
\end{problem}

The set of operators $\L$ is uncountable. One might also add to that
ring the operators $\cD_{\mathfrak{n}}$ and ask the same question. 
\begin{rem}
A similar free module $\text{Fr}_{0}(\PP^{2})$ exists for points
configurations, with the natural ring of operators $\ZZ[\Psi]$. 
\end{rem}

\subsection{\label{subsec:Hirzebruch-inequality-and}Hirzebruch inequality and
operators}

Let $L_{0}$ be an arrangement of $d$ lines in the complex plane.
Let us recall that we denote by $t_{k}(L_{0})$, the number of $k$-points
of $L_{0}$. We recall that a line arrangement is called trivial of
all its lines go through the same point, and quasi-trivial if it is
the union of a trivial arrangement and a line containing only double
points. Hirzebruch inequality is 
\begin{thm}
\label{thm:(Hirzebruch-inequality-for}(Hirzebruch inequality for
lines over $\CC$, \cite[Section 3.1]{Hirz}). Assume that the line
arrangement $L_{0}$ is not trivial nor quasi-trivial, then
\[
t_{2}(L_{0})+t_{3}(L_{0})\geq d+\sum_{k\geq5}(k-4)t_{k}(L_{0}).
\]
\end{thm}

With our notations, one has 
\[
t_{k}(L_{0})=|\mathcal{P}_{\{k\}}(L_{0})|.
\]
Looking at Theorem \ref{thm:(Hirzebruch-inequality-for}, it seems
therefore natural to associate to a line arrangement $L_{0}$, the
formal sum
\[
\begin{array}{cc}
S(L_{0}) & =\mathcal{P}_{\{2\}}(L_{0})+\mathcal{P}_{\{3\}}(L_{0})-\left(\sum_{k\geq5}(k-4)\mathcal{P}_{\{k\}}(L_{0})\right)-\check{\cD}(L_{0})\in\text{Fr}_{0}(\PP^{2})\\
 & =\mathcal{P}_{\{2\}}(L_{0})+\mathcal{P}_{\{3\}}(L_{0})-\left(\sum_{k\geq5}\cP_{k}(L_{0})\right)-\check{\cD}(L_{0}),\hfill
\end{array}
\]
and to consider the natural map 
\[
C:\text{Fr}_{0}(\PP^{2})\to\text{CH}_{0}(\PP^{2}),
\]
which to a point configuration $P\subset\PP^{2}$ associate the element
$\sum_{p\in P}p$ in the Chow group of $\PP^{2}$, and extended linearly
to any element of $\text{Fr}_{0}\PP^{2})$. To prove Hirzebruch inequality
is then equivalent to prove that $C(S(L_{0}))$ has non-negative degree.
Hirzebruch inequality is highly non-trivial, since it is false in
positive characteristic; it uses the Bogomolov-Miyaoka-Yau inequality
for surfaces over $\CC$. 

There are many other Hirzebruch-type inequalities for line arrangements
over $\CC$, (see e.g. \cite{PokoraH}).%
{} Over $\RR$, there is the Melchior inequality:
\begin{thm}
\label{thm:(Melchior)-Let-}(Melchior) Let $\cA$ be an arrangement
of $d\geq3$ real lines not in a pencil. Then 
\[
t_{2}(\cA)\geq3+\sum_{r\geq3}(r-3)t_{r}(\cA).
\]
\end{thm}

It proves in particular that an arrangement of real lines not in a
pencil have always some double points. The proof of Melchior inequality
is combinatorial (see \cite{PokoraH}).

\subsection{\label{subsec:Classical-Theorems}Classical Theorems (Pappus, Pascal,
Desargues...)}

In this section we remark that many classical Theorems of geometry
can be formulated as statements about the operator $\ldt$. In Section
\ref{subsec:Povera-arrangements}, we will give some new results of
classical flavor about that operator, for some arrangements of six
lines. 

Let us recall some vocabulary:
\begin{defn}
An \textit{hexagon} $H$ is the data of a $6$-tuple $(\ell_{1},\dots,\ell_{6})$
of distinct lines and a $6$-tuple $P_{6}=(p_{1},\dots,p_{6})$ of
double points (called the vertices) on the arrangement $H$, such
that $p_{i}$ is the intersection point of $\ell_{i}$ and $\ell_{i+1}$
(the indices being taken mod $6$). The \textit{three diagonals} of
an hexagon are the three lines through the pairs of points $p_{i},p_{i+3}$.
The \textit{opposite sides} of $H$ are the three pairs of lines with
no common points in the set $\{p_{1},\dots,p_{6}\}$. 
\end{defn}

\subsubsection{\label{subsec:Pascal's-Hexagon-Theorem}Pascal's Hexagon Theorem}

Pascal's Hexagon Theorem states that if $H$ is a hexagon such that
its set of vertices $P_{6}$ is contained in an irreducible conic,
then the three intersection points of the three pairs of opposite
sides of $H$ are contained in a line (see Figure \ref{fig:Pascal}).
In other words and using our notation, the arrangement $\ldt(H)$
is one line when the vertices are generic on a conic. That line is
called the Pascal line of the hexagon $H$.

Since one can form $60$ hexagons with vertices in $P_{6}$, there
are $60$ Pascal lines (when $P_{6}$ is generic): that result is
known as the hexagrammum mysticum Theorem. Let $L_{15}=\cP_{\{2\}}(P_{6})$
be the union of the $15$ lines containing two points in $P_{6}$.
Using our notations, the hexagrammum mysticum Theorem means that the
line arrangement 
\[
\L_{\{2\},\{3\}}(L_{15})
\]
is the union of $60$ lines. If $P_{6}$ is the set of the vertices
of the regular hexagon, $\L_{\{2\},\{3\}}(L_{15})$ has $67$ lines;
on random examples, we obtained $60$ lines. 

\begin{figure}[h]
\begin{center}
\begin{tikzpicture}[scale=0.45]

\clip(-4.3,-4.9) rectangle (13.76,3.8);
\draw [rotate around={-9.915526665510901:(3.97,-0.5)},line width=0.2mm] (3.97,-0.5) ellipse (5.388804073211836cm and 3.6667027885369263cm);
\draw [line width=0.2mm,color=blue] (0.858637516031961,2.793676660156259)-- (2.866982604398994,-4.0146097577290725);
\draw [line width=0.2mm,color=blue] (4.996506913938421,3.032444764084481)-- (-0.28436746126893997,-2.344279866038222);
\draw [line width=0.2mm,color=blue] (4.996506913938421,3.032444764084481)-- (6.8537506847184915,-3.8792411213668414);
\draw [line width=0.2mm,color=blue] (2.866982604398994,-4.0146097577290725)-- (8.550987827690065,0.9810469538142258);
\draw [line width=0.2mm,color=blue] (8.550987827690065,0.9810469538142258)-- (-0.28436746126893997,-2.344279866038222);
\draw [line width=0.2mm,color=blue] (0.858637516031961,2.793676660156259)-- (6.8537506847184915,-3.8792411213668414);
\draw [line width=0.2mm,color=red,domain=-4.3:13.76] plot(\x,{(--0.4241016179003205-0.8916103796566817*\x)/4.361311165676245});
\begin{scriptsize}
\draw [fill=black] (0.858637516031961,2.793676660156259) circle (1mm);
\draw[color=black] (1.,3.17) node {$p_1$};
\draw [fill=black] (-0.28436746126893997,-2.344279866038222) circle (1mm);
\draw[color=black] (-0.4,-2.9) node {$p_4$};
\draw [fill=black] (2.866982604398994,-4.0146097577290725) circle (1mm);
\draw[color=black] (3.,-4.65) node {$p_2$};
\draw [fill=black] (4.996506913938421,3.032444764084481) circle (1mm);
\draw[color=black] (5.14,3.41) node {$p_5$};
\draw [fill=black] (8.550987827690065,0.9810469538142258) circle (1mm);
\draw[color=black] (8.7,1.35) node {$p_3$};
\draw [fill=black] (6.8537506847184915,-3.8792411213668414) circle (1mm);
\draw[color=black] (7.,-4.51) node {$p_6$};
\draw [fill=black] (1.7601963617314205,-0.2626062862624834) circle (0.2mm);
\draw [fill=black] (6.121507527407665,-1.154216665919165) circle (0.2mm);
\draw[color=red] (12,-1.6) node {$\Lambda_{\{2\},\{3\}}(H)$};
\draw [fill=black] (4.019432400937969,-0.7244761749565344) circle (0.2mm);
\draw [fill=red] (4.019432400937969,-0.7244761749565344) circle (1.5mm);
\draw [fill=red] (6.121507527407665,-1.154216665919165) circle (1.5mm);
\draw [fill=red] (1.7601963617314205,-0.2626062862624834) circle (1.5mm); 
\end{scriptsize}
\end{tikzpicture}
\end{center} 

\caption{\label{fig:Pascal}Pascal's Hexagon Theorem}
\end{figure}

\subsubsection{\label{subsec:Pappus'-and-Steiner's}Pappus' and Steiner's Theorems}

Pappus' Theorem can be derived from Pascal's Hexagon Theorem by degeneration
of the conic into two lines $\ell_{1},\ell_{2}$. Let $p_{1},p_{2},p_{3}$
(resp. $p_{4},p_{5},p_{6}$) be three points on $\ell_{1}$ (resp.
$\ell_{2}$). Pappus Theorem states that if $H=\{\ell_{3},\dots,\ell_{8}\}$
is an hexagon formed by joining lines though $P_{6}=\{p_{1},\dots,p_{6}\}$,
then there exists a line $\ell_{9}$ (the Pappus line of $H$) such
that $\ldt(H)=\{\ell_{1},\ell_{2},\ell_{9}\}$. When generic, the
Pappus configuration $\{\ell_{1},\dots,\ell_{9}\}$ is a $9_{3}$-configuration:
it has $9$ triple points, each line contains three of these points.
There are six hexagons $H$: the line arrangement $L_{6}=\ldt(\cL_{\{2\}}(P_{6}))$
is the union of the $6$ Pappus lines.

Then Steiner's Theorem says that these six lines are union of two
sets of three lines, each sets of three lines meeting in one point,
and that the operation $\psi_{\{2\},\{3\}}\circ\cP_{\{2\}}$ dual
to $\ldt\circ\cL_{\{2\}}$ applied to $L_{6}$:
\[
P_{6}'=\psi_{\{2\},\{3\}}(\cP_{\{2\}}(L_{6}))=\psi_{\{2\},\{3\}}\circ\cP_{\{2\}}\circ\ldt\circ\cL_{\{2\}}(P_{6})
\]
is again a set $P_{6}'=\{p'_{1},\dots,p'_{6}\}$ of $6$ points, with
$p_{1}',p_{2}',p_{3}'$ (resp. $p_{4}',p_{5}',p_{6}$') on $\ell_{1}$
(resp. $\ell_{2}$). Moreover, the points $p_{1}',p_{2}',p_{3}'$
depend only on the intersection point of $\ell_{1}$ and $\ell_{2}$.
The dynamics of the map $\{p_{1},p_{2},p_{3}\}\to\{p_{1}',p_{2}',p_{3}'\}$
has been studied in \cite{Hooper} (the author thanks Richard Schwartz
for pointing out that reference). We remark that
\[
\psi_{\{2\},\{3\}}\circ\cP_{\{2\}}\circ\ldt\circ\cL_{\{2\}}=\psi_{\{2\},\{3\}}\circ\psi_{\{3\},\{2\}}\circ\psi_{\{2\},\{2\}}.
\]
Adopting the dual view point, Steiner's Theorem tells that the operator
\[
\L_{\{2\},\{3\}}\circ\L_{\{3\},\{2\}}\circ\L_{\{2\},\{2\}}
\]
is a rational self map on the space of six lines concurrent in threes. 

\subsubsection{\label{subsec:Desargues'-Theorem}Desargues' Theorem}

Let $\ell_{1},\ell_{2},\ell_{3}$ be three lines meeting at the same
point, let $\ell_{1}',\ell_{2}',\ell_{3}'$, (resp. $\ell_{1}'',\ell_{2}'',\ell_{3}''$)
be lines such that their three intersection points are on lines $\ell_{1},\ell_{2},\ell_{3}$
(see Figure \ref{fig:Desargues-configuration}). Let $\cL_{9}$ be
the union of the nine lines $\ell_{1},\dots,\ell_{3}''$. Desargues'
Theorem tells that $\ldt(\cL_{9})$ is one line. 

\begin{figure}[h]
\begin{center}
\begin{tikzpicture}[scale=0.23]
\clip(-11.46,-15.2) rectangle (16.585,9.02125);
\draw [line width=0.2mm,color=blue,domain=-10.46:16.585] plot(\x,{(--8.--4.*\x)/8.});
\draw [line width=0.2mm,color=blue,domain=-10.46:16.585] plot(\x,{(-0.-0.*\x)/6.});
\draw [line width=0.2mm,color=blue,domain=-10.46:16.585] plot(\x,{(-6.6175-3.30875*\x)/8.055});
\draw [line width=0.2mm,domain=-10.46:16.585] plot(\x,{(--16.-4.*\x)/-2.});
\draw [line width=0.2mm,domain=-10.46:16.585] plot(\x,{(--13.235-3.30875*\x)/2.055});

\draw [line width=0.2mm] (6.1,-16.22375) -- (6.,9.02125);
\draw [line width=0.2mm,color=green,domain=-10.46:16.585] plot(\x,{(--0.75101-1.44425*\x)/-0.3685});
\draw [line width=0.2mm,color=green,domain=-10.46:16.585] plot(\x,{(--4.3981290625-3.098625*\x)/1.139});
\draw [line width=0.2mm,color=green,domain=-10.46:16.585] plot(\x,{(-0.860275--1.654375*\x)/-1.5075});
\draw [line width=0.2mm,color=red,domain=-10.46:16.585] plot(\x,{(-126.85289683417939--1.4072456639447122*\x)/9.232849243353375});
\begin{scriptsize}
\draw [fill=blue] (-2.,0.) circle (2mm);
\draw [fill=black] (6.,4.) circle (2mm);
\draw[color=blue] (-9.83,-3.00125) node {$\ell_1$};
\draw [fill=black] (4.,0.) circle (2mm);
\draw[color=blue] (-9.83,-0.81625) node {$\ell_2$};
\draw [fill=black] (6.055,-3.30875) circle (2mm);
\draw[color=blue] (-9.83,2.56875) node {$\ell_3$};
\draw[color=black] (9,8.32375) node {$\ell_3'$};
\draw[color=black] (-0.205,8.32375) node {$\ell_2 '$};
\draw[color=black] (5.205,8.32375) node {$\ell_1 '$};
\draw [fill=green] (0.8885,1.44425) circle (2mm);
\draw [fill=green] (0.52,0.) circle (2mm);
\draw[color=green] (3.58,8.32375) node {$\ell_1''$};
\draw [fill=green] (2.0275,-1.654375) circle (2mm);
\draw[color=green] (-2.5,8.32375) node {$\ell_2''$};
\draw[color=green] (-6.,8.32375) node {$\ell_3''$};
\draw [fill=red] (-3.1063838812301174,-14.212767762460235) circle (3mm);
\draw [fill=red] (6.126465362123257,-12.805522098515523) circle (3mm);
\draw[color=red] (-7,-13.5) node {$\Lambda_{\{2\},\{3\}}(\mathcal{L}_9)$};
\draw [fill=red] (11.449374999999998,-11.99421875) circle (3mm);
\end{scriptsize}

\end{tikzpicture}
\end{center} 

\caption{\label{fig:Desargues-configuration}Desargues configuration}
\end{figure}

The union $\cL_{10}$ of $\cL_{9}$ and the Desargues line $\ldt(\cL_{9})$
is a $10_{3}$-configuration: $\cL_{10}$ has $10$ triple points
and $10$ lines such that each line contains $3$ triple points. In
fact $\cL_{10}=\L_{2,3}(\cL_{9})$; moreover $\L_{3}(\cL_{10})=\cL_{10}$. 

\subsubsection{\label{subsec:Brianchon's-Theorem}Brianchon's Theorem}

Brianchon's Theorem is the dual of Pascal's Theorem, it states that
when a hexagon $H$ with vertices $P_{6}$ is circumscribed around
a conic, then the three diagonals of the hexagon meet in a single
point, in other words: $\Psi_{\{2\},\{3\}}(P_{6})$ is a point. 

The dual of Pascal's hexagrammum mysticum Theorem may be formulated
by saying that the point arrangement
\[
\cP_{\{3\}}(\L_{\{2\},\{2\}}(H))=\Psi_{\{2\},\{3\}}(\cP_{2}(H))
\]
is a set of (at least) $60$ points.

\subsection{\label{subsec:Green-and-Tao}Green and Tao results }

In \cite{GT}, Green and Tao obtain the following results:
\begin{thm}
(Orchard problem ; \cite[Theorem 1.3]{GT} of Green-Tao). Suppose
that $\cP$ is a finite set of $n$ points in the real plane. Suppose
that $n\geq n_{0}$ for some sufficiently large absolute constant
$n_{0}$. Then there are no more than $\left\lfloor n(n-3)/6\right\rfloor +1$
lines that are $3$-rich, that is they contain precisely $3$ points
of $\cP$.
\end{thm}

For a line arrangement $\cC$, let us denote by $|\cC|$ the number
of lines of $\cC$. The above result be be rephrased as follows: 
\begin{thm*}
(Green-Tao). Let $\cP$ be a set of real points with $n=|\cP|\geq n_{0}$
for some sufficiently large absolute constant $n_{0}$. Then $|\cL_{\{3\}}(\cP)|\leq\left\lfloor n(n\text{\textminus}3)/6\right\rfloor +1$. 
\end{thm*}
Strongly related to that problem is the following result
\begin{thm}
(Dirac--Motzkin conjecture ; \cite[Theorem 1.4]{GT} of Green-Tao).
Suppose that $\cP$ is a finite set of $n$ points in the real plane,
not all on one line. Suppose that $n\geq n_{0}$ for a sufficiently
large absolute constant $n_{0}$. Then $\cP$ spans at least $n/2$
ordinary lines.
\end{thm}

In other words: 
\begin{thm*}
(Green-Tao). There exists $n_{0}\in\NN$ such that for a set of points
$\cP$ not all on one line with $|\cP|\geq n_{0}$, one has $|\cL_{\{2\}}(\cP)|\geq|\cP|/2$.
\end{thm*}
There is an infinite number of point arrangements $\cP$ such that
$|\cL_{\{2\}}(\cP)|=|\cP|/2$, these arrangements are due to Böröczky.
The two known point arrangements such that $|\cL_{\{2\}}(\cP)|<|\cP|/2$
are i) the $7$ points which are the vertices of a triangle, the mid-points
and the center of that triangle, and ii) an arrangement of $13$ points
due independently to Böröczky and McKee. %
{} As Green and Tao remark in their introduction of \cite{GT}, the
dual of point arrangements with few ordinary lines are line arrangement
containing many triangles. In fact the dual of the Böröczky points
arrangements in i) and ii) of \cite[Proposition 2.1]{GT} are the
two known infinite families of simplicial line arrangements which
are not quasi-trivial (see Section \ref{subsec:Simplicial-Arrangements-and}).

\section{\label{sec:Some-examples}Action of operators on some examples of
line arrangements }

\subsection{\label{subsec:Some-classical-arrangements}Some classical arrangements}

In this section we give some examples of line arrangements and how
is the action of some operators on them. 

Before that, let us recall the definition of $H$-constant which is
used in order to study the Bounded Negativity Conjecture. The $H$-constant
of a line arrangement $L_{0}$ of $d$ lines has been defined in \cite{IMRN}
by the formula 
\[
H=\frac{d^{2}-\sum_{m\geq2}m^{2}t_{m}}{\sum_{m\geq2}t_{m}}.
\]
It is proved in \cite{IMRN} (using inequalities of Hirzebruch) that
over a field of characteristic zero, $-4$ is a lower bound for the
$H$-constants of lines. Line arrangements with low $H$-constant
have remarkable properties, for example the line arrangement with
lowest known $H$-constant ($\simeq-3.3582$) is the Wiman configuration
of $45$ lines constructed from some sporadic complex reflection group.
Any known arrangements with $H$-constant $<-3$ were obtained by
removing some lines from the Wiman arrangement.

\subsubsection{\label{subsec:Trivial,-quasi-trivial-and}Trivial, quasi-trivial
and Finite plane arrangements}

\begin{example}
(\textit{The trivial arrangements}). A line arrangement is \textit{trivial}
if this is a union of a finite number $n\geq0$ of lines in the same
pencil i.e. of lines going through the same point. If $n\geq2$, this
is equivalent to $t_{n}=1$. Then for any line operator, $\L(L_{0})=\emptyset$. 
\end{example}

Conversely:
\begin{prop}
Let $L_{0}$ be a line arrangement such that $\L(L_{0})=\emptyset$
for any line operator $\L$. Then $L_{0}$ is trivial.
\end{prop}

\begin{proof}
If $L_{0}$ is a line arrangement such that $\L_{2}(L_{0})=\emptyset$,
then $\cP_{2}$ is either empty or a point, therefore $L_{0}$ is
a trivial arrangement.
\end{proof}
\begin{example}
(\textit{The quasi-trivial arrangements}) A line arrangement of $n\geq3$
lines is \textit{quasi-trivial} if this is a union of a trivial arrangement
and a line not going through the intersection point of the lines in
the pencil. When $n\geq4$, this is equivalent to the equality $t_{n-1}=1$.
The $\L_{2}$-sequence associated to a quasi-trivial arrangement is
constant (and non-extinguishing). 
\end{example}

\begin{example}
(\textit{Finite plane arrangements}). Let us suppose that the base
field $K$ is the closure of a finite field $\FF_{q}$. Let $L_{0}\in\check{\PP^{2}(\FF_{q})}$
be a line arrangement. For a fixed line operator $\L$, let us consider
the $\L$-sequence $(L_{m})_{m}$ associated to $L_{0}$. For all
$m\geq0$, the intersection points of $L_{m}$ are defined over $\FF_{q}$
and therefore the lines in $L_{m+1}$ are in $\check{\PP}^{2}(\FF_{q})$.
Since these are finite sets, the $\L$-sequence is always periodic.
That simple example shows that proving that a sequence $(L_{n})_{n\geq0}$
associated to a line operator $\L$ and an arrangement $L_{0}$ acquire
more and more lines is a non-trivial result over a field of characteristic
$0$.\\
The arrangement $L_{0}=\check{\PP}^{2}(\FF_{q})$ is called a finite
plane arrangement. If $n,m\leq q+1$, the $\L_{n,m}$-sequence associated
to $\check{\PP}^{2}(\FF_{q})$ is constant. 
\end{example}

Quasi-trivial arrangements and finite plane arrangements are two examples
of line arrangements $L_{0}$ such that $\L_{2}(L_{0})=L_{0}$. Conversely:
\begin{prop}
\label{prop:Lambda2stable}Let $L_{0}\neq\emptyset$ be an arrangement
of lines. Suppose that $\L_{2}(L_{0})=L_{0}$. Then $L_{0}$ is a
quasi-trivial arrangement or a finite projective plane.
\end{prop}

Proposition \ref{prop:Lambda2stable} is an application of the following
result:
\begin{thm}
\label{thm:(De-Bruin-=002013}(De Bruijn--Erdös). Let $L_{0}$ be
an arrangement of $d\geq3$ lines which is not a trivial arrangement
and let $t$ be the number of singular points. Then $t\geq d$ and
$t=d$ if and only if $L_{0}$ is a quasi-trivial arrangement or a
finite projective plane arrangement. 
\end{thm}

For a proof of Theorem \ref{thm:(De-Bruin-=002013}, we refer to \cite[Theorem 2.7]{EFU}.
\begin{proof}
(Of Proposition \ref{prop:Lambda2stable}). Let $P_{0}=\{p_{1},\dots,p_{t}\}=\cP_{2}(L_{0})$
be the set of singularities of the line arrangement $L_{0}=\{\ell_{1},\dots,\ell_{d}\}$.
By the Theorem of De Bruijn--Erdös, we have $t\geq d$. 

Let us denote by $[\ell]$ the point in the dual plane corresponding
to a line $\ell$, and by $^{t}p$ the line in the dual plane defined
by a point $p$ in the plane. Since $\L_{2}(L_{0})=L_{0}$, for any
two points $p\neq q\in P_{0}$, the line $\ell$ passing through $p$
and $q$ is contained in $L_{0}$. On the dual plane level, that implies
that any two lines $^{t}p,^{t}q$ in the line arrangement $\check{P_{0}}=\{^{t}p_{1},\dots,^{t}p_{t}\}$
meet in a point $[\ell]$ such that $\ell\in L_{0}$. Therefore the
set $\{[\ell],\,|\,\ell\in L_{0}\}$ is the singularity set of $\check{P_{0}}$,
and by the theorem of De Bruijn--Erdös, one get $d\geq t$. We
obtain that $d=t$, and the result follows again from Theorem \ref{thm:(De-Bruin-=002013}. 
\end{proof}

\subsubsection{\label{subsec:Classical-arrangements}Classical arrangements}

A reference for most of the line arrangements presented here is \cite[Chapter 6]{GR}. 
\begin{example}
\label{exa:(The-complete-quadrilateral).}(\textit{The complete quadrilateral}).
Let $L_{0}$ be the arrangement of six lines going through four points
in general position.\\
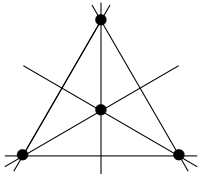
\begin{center}
\begin{tikzpicture}[scale=1.2]



\draw  ( 0,-0.7 )  -- (0, 1.2); 
\draw  ( -0.866,-0.5 )  -- (0,1); 
\draw  ( -0.966,-0.67 )  -- (0.1,1.173); 
\draw  ( 0.966,-0.67 )  -- (-0.1,1.173); 
\draw  ( -1.066,-0.5 )  -- (1.066,-0.5); 
\draw  ( -1.066,-0.62 )  -- (0.86,0.5); 
\draw  ( 1.066,-0.62 )  -- (-0.86,0.5); 

\draw ( 0, 0 ) node {$\bullet$};
\draw ( 0, 1 ) node {$\bullet$};
\draw ( 0.866,-0.5 ) node {$\bullet$};
\draw ( -0.866,-0.5 ) node {$\bullet$};


\end{tikzpicture}
\end{center} 
 Then $P_{0}=\cP_{2}(L{}_{0})$ is a set of $7$ points and $\cL_{3}(\cP_{2}(L_{0}))=L_{0}$,
i.e. $L_{0}$ is fixed by $\L_{2,3}$, unless in characteristic $2$,
for which $\L_{3,2}(L_{0})=\L_{3}(L_{0})$ is the Fano plane, i.e.
the $7$ lines in $\PP^{2}(\FF_{2})$, giving the unique $7_{3}$-configuration.
In characteristic $0$, the first terms of the associated $\L_{2}$-sequence
$(L_{n})_{n}$ are such that \\
\begin{tabular}{|c|c|c|c|c|c|c|}
\hline 
 & $|L_{n}|$ & $H_{cst}$ & $t_{2}$ & $t_{3}$ & $t_{4}$ & $t_{6}$\tabularnewline
\hline 
$L_{0}$ & $6$ & $-1,714$ & $3$ & $4$ &  & \tabularnewline
\hline 
$L_{1}$ & $9$ & $-2.077$ & $6$ & $4$ & $3$ & \tabularnewline
\hline 
$L_{2}$ & $25$ & $-2.464$ & $60$ & $24$ & $3$ & $10$\tabularnewline
\hline 
\end{tabular}\\
where the column $H_{cst}$ is an approximation of the $H$-constant;
the line arrangement $L_{3}$ has $1471$ lines. The number of lines
in $L_{m}$ grows to infinity, as shown in \cite{CW}. It would be
interesting to know the limit of the $H$-constants of the configurations
$L_{n}$. 
\end{example}

\begin{example}
(\textit{\label{exa:(Hesse-and-dual}Hesse and dual Hesse arrangements}).
The Hesse arrangement $\cH$ is defined over fields containing a third
root of unity $\o\neq1$. It is the union of $12$ lines and its set
$P_{0}$ of $21$ singular points are such that:
\[
t_{2}=12,\,t_{4}=9.
\]
The nine $4$-points and the $12$ lines form a $(9_{4},12_{3})$-configuration
(see Section \ref{subsec:The-(r,n)-configurations} for the definition
of a $(r_{k},s_{m})$-configuration). One has $\cL_{3}(P_{0})=\cH$,
$\cP_{3}(\cH)=P_{0}$, thus $\L_{3}(\cH)=\cH$ and the associated
$\L_{3}$-sequence is constant. The Hesse arrangement is $\L_{\{4\},\{3\}}$-constant
and $\cL_{\{2\}}(\mathcal{D}(\cH))$ is projectively equivalent to
$\cH$. 

The dual Hesse arrangement $\check{\cH}=\cL_{\{4\}}(\mathcal{D}(\cH))$
is a $(12_{3},9_{4})$-configuration of $9$ lines with equations
\[
\begin{array}{c}
-x+z,-x+y,-y+z,-\om x+z,-\om x+y,\\
-\om^{2}x+z,-\om^{2}x+y,-\om^{2}y+z,-\om y+z,
\end{array}
\]
where $\o^{2}+\o+1=0$. The arrangement $\check{\cH}$ has $12$ triple
points 
\[
\begin{array}{c}
(1:0:0),(0:1:0),(0:0:1),(1:1:1),\\
(1:1:\om),(1:1:\om^{2}),(1:\om:1),(1:\om^{2}:1),\\
(\om:1:1),(\om^{2}:1:1),(\om:\om^{2}:1),(\om^{2}:\om:1),
\end{array}
\]
it is the unique known arrangement over $\CC$ which has only triple
points. One has $\L_{3}(\check{\cH})=\check{\cH}$, and $\L_{3,4}(\check{\cH})=\check{\cH}$.
The line arrangement $\L_{\{2\},\{4\}}(\cH)$ is projectively equivalent
to $\check{\cH}$. 

The arrangement $\text{ML}$ obtained by removing one line in $\check{\cH}$
is a $8_{3}$-arrangement known as the \textit{MacLane} or \textit{Möbius-Kantor}
arrangement. The arrangement $\L_{3,2}(\text{ML})$ is isomorphic
to the Hesse arrangement $\cH$. The line arrangement $\L_{3,2}(\check{\cH})$
is an arrangement of $21$ lines, it is the union of $\check{\cH}$
and a line arrangement projectively equivalent to $\cH$. Its singularities
are $t_{2}=36,t_{4}=9,t_{5}=12$ and it is a free arrangement (see
\cite[Proposition 4.6]{Ilardi}).

The (complex) line arrangements $\cH$ and $\L_{3,2}(\check{\cH})$
are combinatorially simplicial i.e. their singularities satisfy Equation
\eqref{eq:Simplicial}, cf. later in Section \ref{subsec:Simplicial-Arrangements-and}
the definition of simplicial line arrangement.
\end{example}

\begin{example}
(\textit{Ceva$(n)$ arrangement}) The Ceva$(n)$ arrangement is given
by the $3n$ lines 
\[
(x^{n}-y^{n})(x^{n}-z^{n})(y^{n}-z^{n})=0.
\]
Ceva$(2)$ is the complete quadrilateral, Ceva$(3)$ is the dual Hesse
arrangement. The arrangement Ceva$(n)$ has $n^{2}$ triple points,
three $n$-points and no other singularities; the $H$-constant is
$-\frac{3n^{2}}{n^{2}+3}$. Each line contain one $n$-point and $n$
triple points, thus for $n\geq3$, the $\Lambda_{3}$-sequence is
constant. In characteristic $0$, this is the unique known infinite
family of \textit{Sylvester-Gallai} arrangements, i.e. line arrangements
with no double points. These are also reflection arrangements, i.e.
the lines are the fixed lines of some reflexion group (denoted $G(n,n,3)$)
acting on the plane. 

The arrangement $\overline{\text{Ceva}(n)}$ obtained as the union
of $\text{Ceva}(n)$ and the three lines $xyz=0$ is a free arrangement.
Each of the three lines add up $n$ double points, and the three $n$-points
of $\text{Ceva}(n)$ become $n+2$-points, so that the singularities
are 
\[
t_{2}=3n,t_{3}=n^{2},t_{n+2}=3.
\]
These line arrangements are combinatorially simplicial and one has
$\L_{3}(\overline{\text{Ceva}(n)})=\text{Ceva}(n)$. 
\end{example}

\begin{example}
(\textit{Polygonal arrangements} $A_{1}(2m)$, \textit{extended Polygonal
arrangements} $A_{1}(4m+1)$, see \cite{Hirz}). For $m\geq3$, the
regular arrangement $A_{1}(2m)$ consists of $2m$ lines, of which
$m$ lines are determined by the edges of a regular $m$-gon in the
Euclidean plane (contained in $\PP^{2}$), while the other $m$ are
the lines of symmetry of that $m$-gon. For $m=3$, we get the complete
quadrilateral; for $m>3$, one has 
\[
t_{2}=m,\,t_{3}=\frac{m(m-1)}{2},\,t_{m}=1,\,t_{r}=0\text{ otherwise.}
\]
Let $P_{0}=\cP_{3}(A_{1}(2m))$ be the set of triple or higher multiplicity
points of the $A_{1}(2m)$ arrangement. \\
$\bullet$ Suppose that $m$ is odd. Let us show that $\L_{3}(A_{1}(2m))=A_{1}(2m)$
when $m\geq5$. Let us write $m=2k+1$. Then a line of symmetry contains
$k$ triple points (intersection by pairs of $2k$ lines of the faces)
and one double point (intersection of the line with its opposite face).
We get in that way $km=\frac{m(m-1)}{2}$ triple points and $m$ double
points. A face contains $m-1$ triple points, intersection of $m-1$
symmetry lines, and the $m-1$ other faces. We get $m(m-1)$ triple
points, but each point is counted twice, so that the total number
of triple points is $\frac{m(m-1)}{2}$. Therefore, as soon as $m$
is odd and $m\geq5$, the associated $\L_{3}$-sequence is constant.
\\
$\bullet$ Suppose that $m$ is even, $m=2k$. Let us show that when
$k\geq3$, 
\[
\L_{3}(A_{1}(4k))=A_{1}(4k+1),\,\,\L_{3}(A_{1}(4k+1))=A_{1}(4k+1),
\]
where $A_{1}(4k+1)$ is the union of $A_{1}(4k)$ and the line at
infinity, its singularities are 
\[
t_{2}=3k,\,t_{3}=2k(k-1),\,t_{4}=k,\,t_{2k}=1,\,t_{r}=0
\]
otherwise. The arrangement $A_{1}(4k)$ has $2k$ faces, $k$ lines
of symmetries each containing two edges of the $2k$-gon, $k$ lines
of symmetries, each containing two double points, which are in the
middle between to consecutive edges. \\
The number of double points on a face is $1$ and there are $2k-1$
triple points: each face meet the other $2k-1$ faces and $2k-1$
lines of symmetries into $2m-1$ triple points, the remaining line
of symmetry cuts the face in one double point. \\
A line of symmetry containing two edges cuts the $2k-1$ other lines
of symmetry at the same point, a point of multiplicity $2k$, and
cuts the $2k$ lines of the $2k$-gon in $k$ triple points. \\
A line of symmetry containing two double points contains also the
$2m$-point and $k-1$ triple points, intersection points of the $2k$
faces. \\
The $A_{1}(4k+1)$-configuration is obtained by adding the line at
infinity, which contains then $k$ $4$-points, which were previously
triple points on $A_{1}(4k)$ coming from two parallel faces and the
parallel line of symmetry, and moreover $k$ double points, intersection
of the $k$ lines containing $2$ edges of the $2k$-gon.\\
Thus when $k-1\geq2$, we obtain that all lines in $A_{1}(4k)$ contain
$3$ or more $m$-points with $m\geq3$. \\
For $k=2$, i.e. $A_{1}(9)$, the singularities are $t_{2}=6,\,t_{3}=4,t_{4}=3$
and we get that $\L_{3}(A_{1}(9))$ is the complete quadrilateral
$A_{1}(6)$. 
\end{example}

\begin{example}
(\textit{Klein configuration}) %
{} The Klein configuration $\cK$ of $21$ lines is a reflexion arrangement:
it is the union of the hyperplanes fixed by the $21$ involutions
of the unique simple group of order $168$ (acting faithfully on $\PP^{2}$).
The singularity set of $\cK$ is such that
\[
t_{3}=28,t_{4}=21,\,t_{r}=0\text{ for }r\neq3,4.
\]
The arrangement $\cK$ is defined over fields containing $\sqrt{-7}$;
the $H$-constant is $H(\cK)=-3$. The $21$ quadruple points and
$21$ lines form a $(21_{4})$-configuration (see Section \ref{subsec:The-(r,n)-configurations}
for the definition. Each lines contains $4$ triple points and $4$
quadruple points. The configuration is $\L_{4}$-constant: $\L_{4}(\cK)=\cK$.
It is also a fixed arrangement for $\L_{3,4}$, but not for $\L_{3}$
nor for $\L_{4,3}$.

The arrangement $\L_{3}(\cK)$ has $133$ lines, and singularities
\[
t_{2}=2436,\,t_{3}=588,\,t_{4}=84,\,t_{5}=168,\,t_{9}=28,\,t_{12}=21.
\]
The $9$-points and $12$-points of $\L_{3}(\cK)$ are the singular
points of the Klein configuration. \\
The arrangement $\L_{4,3}(\cK)$ has $49$ lines, and singularities
\[
t_{2}=252,\,t_{3}=112,\,t_{8}=21,
\]
it is fixed by $\L_{4,3}$, one has $\L_{4}(\L_{4,3}(\cK))=\cK$,
moreover $\L_{3}(\L_{4,3}(\cK))$ contains $889$ lines. %
\end{example}

\begin{example}
 (\textit{Grünbaum-Rigby configuration}, see \cite{GR}, see also
\cite{GP}) The Grünbaum-Rigby configuration $\cC=(P_{0},L_{0})$
of $91$ points and $21$ lines is such that 
\[
t_{2}=63,\,t_{3}=7,\,t_{4}=21,
\]
with $H_{cst}=-30/13\simeq-2.30$. The $21$ $4$-points and the $21$
lines form a $(21_{4})$-configuration which, as an abstract configuration,
is isomorphic to the $(21_{4})$-configuration of the Klein arrangement
(see Section \ref{subsec:The-(r,n)-configurations} for the definition
of a $(r_{k},s_{m})$-configuration). The $63$ double points with
the $21$ lines form a $(63_{2},21_{6})$-configuration, moreover
each line contains a unique triple point. One has $\L_{4}(L_{0})=L_{0}$:
this is a fixed configuration for $\L_{4}$. \\
Let $P_{1}=\cP_{3}(L_{0})$ be the $28$ points with multiplicity
$\text{\ensuremath{\geq3}}$. The arrangement $\cL_{3}(P_{1})=\L_{3}(L_{0})$
has $50$ lines, and the singularities are 
\[
t_{2}=259,\,t_{3}=119,\,t_{7}=29.
\]
The points in $P_{1}$ are now points of multiplicity $7$ in $\L_{3}(L_{0})$.
The arrangement $\L_{3}(L_{0})$ is fixed by $\L_{4,3}$. \\
The arrangement $\Lambda_{4}(\L_{3}(L_{0}))$ has $29$ lines, singularities
\[
t_{2}=70,t_{3}=21,t_{4}=7,t_{5}=21,t_{7}=1,
\]
and $H$-constant $-157/60$. It is a fix arrangement for $\L_{4}$.
\\
The arrangement $\cL_{4}(P_{1})$ has $28$ lines, the $H$-constant
is $H=-294/113\simeq-2.6$, and singularities 
\[
t_{2}=63,\,t_{3}=28,\,t_{5}=21,\,t_{7}=1.
\]
This arrangement is fixed for $\L_{4}$; the $28$ lines and $28$
triple points form a $28_{3}$-configuration. 
\end{example}

\begin{example}
(\textit{The Wiman configuration}) The Wiman configuration $\cW$
is a reflexion arrangement, it is the union of $45$ lines defined
over a field containing a $15$-th root of unity. It's singularities
are 
\[
t_{3}=120,\,t_{4}=45,\,t_{5}=36.
\]
 The $45$ lines and $45$ $4$-points form a $45_{4}$-configuration
and the $45$ lines and $36$ $5$-points form a $(36_{5},45_{4})$-configuration;
with the $120$ triple point that gives a $(120_{3},45_{8})$-configuration.
The configuration $\cW$ is fixed by the $\L_{5,4}$ operator. It's
$H$-constant is $H(\cW)=-225/67\simeq-3.3582$, this is the lowest
$H$-constant known for line arrangements. 

The configuration $\L_{4,5}(\cW)$ has $81$ lines and singularities
\[
t_{2}=180,t_{3}=480,t_{5}=36,t_{8}=45,
\]
and $H(\L_{4,5}(\cW))=-753/247\simeq-3.0486<-3$. It is fixed by the
operator $\L_{4,5}$. %
\end{example}

\subsection{\label{subsec:The-(r,n)-configurations}The $(r_{k},s_{m})$-configurations}

\subsubsection{Definition of $(r_{k},s_{m})$-configurations}

For integers $r,s,k,m$, a $(r_{k},s_{m})$-configuration $(\cQ,\cC)$
is an arrangement $\cC$ of $s$ lines and a set of $r$ singular
points $\cQ$ of $\cC$ such that each line of $\cC$ contains exactly
$m$ points of $\cQ$ and each point of $\cQ$ is incident to exactly
$k$ lines of $\cC$. One has $rk=sm$, moreover, if $r=s$, one simply
speaks of a $r_{k}$-configuration. If the numbers $r,s$ are unimportant
for the problem considered, one speaks of a $[k,m]$-configuration
(and of a $k$-configuration if $k=m$). References for points and
lines configurations are for example \cite{Dolgachev}, \cite{Grunbaum}.
Note that in the definition of configuration, there can exists other
$m$-rich lines and other $k$-points. 
\begin{lem}
\label{lem:If-an-arrangement}Let $\cC$ be a $\L_{\{k\},\{m\}}$-constant
arrangement. Then $(\cP_{\{k\}}(\cC),\cC)$ is a $[k,m]$-configuration.
\\
Let $(\cQ,\cC)$ be a $[k,m]$-configuration. Then each line of $\cC$
is a line of $\L_{\{k\},\{m\}}(\cC)$. 
\end{lem}

\begin{proof}
Suppose that $\cC$ is a line arrangement such that $\L_{\{k\},\{m\}}(\cC)=\cC$.
That means that each line of $\cC$ contains $m$ points in $\cP_{\{k\}}(\cC)$,
and therefore $(\cP_{\{k\}}(\cC),\cC)$ is a $[k,m]$-configuration.

Suppose that $(\cQ,\cC)$ is a $[k,m]$-configuration. If $\ell$
is a line of $\cC$, then $\ell$ contains $m$ exactly $k$-points
i.e. $m$ points in $\cP_{\{k\}}(\cC)$, thus $\ell$ is line of $\L_{\{k\},\{m\}}(\cC)$.
 
\end{proof}
It may happen that for a $[k,m]$-configuration $\cC$, the line arrangement
$\L_{\{k\},\{m\}}(\cC)$ contains more lines, as the following examples
show:
\begin{example}
A) consider the arrangement $\cC$ of $6$ lines with equations $x,x+z,x-z,y,y+z,y-z$.
The $9$ points $(a:b:1)\in\PP^{2}$ with $a,b\in\{-1,0,1\}$ are
the double points of $\cC$, and there are two triple points at infinity.
The $9$ double points and the $6$ lines form a $(9_{2},6_{3})$-configuration.
However, the arrangement $\L_{\{2\},\{3\}}(\cC)$ contains $\cC$
and the two lines with equation $x-y,x+y$. 

B) The Pappus $9_{3}$-configuration $\text{P}$ has also nine double
points which with the $9$ lines form a $9_{2}$-configuration, moreover
$\L_{\{2\},\{2\}}(\text{P})$ has $36$ lines. 

C) Consider a $k$-configuration $\cC$, then using appropriate projective
transformations, one can take $k$ projectively equivalent copies
$\cC_{1}$,...,$\cC_{k}$ of $\cC$ in $\PP^{2}$ with one $k$-point
$p_{i}$ on each $\cC_{i}$ in such a way that $p_{1},\dots,p_{k}$
are contained on a line not in $\cup\cC_{i}$, then that line is in
$\L_{\{k\}}(\cC)$ but is not in $\cC=\cup\cC_{i}$.
\end{example}

Let $(\cQ,\cC)$ and $(\cQ',\cC')$ be two $[k,m]$-configurations.
Up to using a projective automorphism of $\PP^{2}$, let us suppose
that $\cQ\cap\cQ'=\emptyset$. Then the union $(\cQ\cup\cQ',\cC\cup\cC')$
is a $[k,m]$-configuration. That leads to the following definitions:
\begin{defn}
A $[k,m]$-configuration $(\cQ,\cC)$ is said \textit{connected} if
for any two points $p,p'\in\cQ$, there exists a sequence of $k$-points
$p_{1}=p,p_{2},\dots,p_{j}=p'$ such that $p_{i},p_{i+1}$ are two
$k$-points on the same line of $\cC$.
\end{defn}

\begin{problem}
We do not know if there exists connected $[k,m]$-configurations with
$k>2$ such that $\L_{\{k\},\{m\}}(\cC)\neq\cC$.
\end{problem}

The dual arrangement of a $[k,m]$-configuration $(\cQ,\cC)$ is defined
by $\check{\cC}=\cD(\cQ)$ and $(\cP_{m}(\check{C}),\check{C})$ is
a $[m,k]$-configuration. In particular, a $n_{k}$-configuration
$(\cQ,\cC)$ has the feature that its dual is a $n_{k}$-configuration;
it is called self dual if the two configurations are isomorphic i.e.
if there exists an isomorphism between planes that sends one arrangement
into the other. 

\subsubsection{Examples of $(r_{k},s_{m})$-configurations}

An arrangement $\cC$ of $n\geq3$ lines in general position is a
$(\left(\frac{n(n-1)}{2}\right)_{2},n_{n-1})$-configuration, and
$\L_{2,n-1}(\cC)=\cC$; one even has the relation $\L_{\{2\},\{n-1\}}(\cC)=\cC$. 

Another example of $\L_{\{2\},\{m\}}$-constant arrangements with
$m\geq3$ is as follows: Let $\cC$ be the union of two pencils of
$m$ lines in general position. Each line of $\cC$ contains exactly
$m$ double points, thus $\cC$ is $\L_{\{2\},\{m\}}$-constant. More
generally, if $\cC$ is the union of $r>1$ pencils of $s$ lines
in general position, then $\cC$ is $\L_{\{2\},\{(r-1)s\}}$-constant.

These arrangements have the property that each line contain a unique
$k$-point with $k>2$. But for example, it is possible to construct
a $\L_{\{2\},\{4\}}$-constant arrangement of $9$ lines with $6$
triple points and $18$ double points such that each line contain
$2$ triple points and $4$ double points.

We have seen that the Fano plane $\PP^{2}(\FF_{2})$ is a $7_{3}$-configuration.
There exist $n_{3}$-configurations in characteristic $\neq2,3$ if
and only if $n\geq8$. 

Consider the dual Hesse $(12_{3},9_{4})$-configuration $\check{\cH}$
(which exists over fields containing a third root of unity). Removing
one line and the four points on that line gives the the Möbius-Kantor
$8_{3}$-configuration $\text{MK}$. One has $\L_{3}(\text{MK})=\text{MK}$,
and $\L_{2,3}(\text{MK})=\check{\cH}$. 

The Pappus configuration (see Section \ref{subsec:Pappus'-and-Steiner's})
is a $9_{3}$-configuration with singularities $t_{2}=9,t_{3}=9$;
it is $\L_{2,3}$ and $\L_{3}$-constant. There are two other non
equivalent $9_{3}$-configurations, they are also $\L_{3}$-constant.
These configurations are defined over the rationals. 

It is known that there exist $n_{4}$-configurations of real lines
for $n=18$ and $n\geq20$ except possibly for $n\in\{23,37,43\}$.
There exists no $19_{4}$-configuration over the real field. There
are only two real $18_{4}$-configurations (over $\RR$), the first
was found by Bokowski and Schewe (in \cite{BS}) and is defined over
$\QQ(\sqrt{5})$, the second was found by Bokowski and Pilaud (see
\cite{BP}) and is defined over $\QQ(\sqrt[3]{108+12\sqrt{93}})$.

\subsubsection{\label{subsec:-arrangements}$2$-arrangements}

Let us define a $[k,m]$-arrangement as a $[k,m]$-configuration $(\cQ,\cL)$
such that $t_{k}(\cL)=\cQ$: a $n_{2}$-arrangement is therefore an
arrangement of $n$ lines such that each line contains exactly two
double points. 

The $k$-configurations for $k\geq3$ have been extensively studied,
and it is easy to construct $n_{2}$-configurations by discarding
some nodal points on $n$ lines in general position. But if one is
rather interested by $2$-arrangements, much less seems to be known.
Here is the list of $2$-arrangements we are aware of:
\begin{enumerate}
\item The triangle: three lines with $t_{2}=2$.
\item The three $9_{3}$-configurations and their nodes are also $9_{2}$-arrangements.
\item The Hesse arrangement is a $12_{2}$-arrangement.
\item The $15_{2}$ icosahedral line arrangement.
\end{enumerate}
If one imposes no triple points, one gets the following result, which
the author owes to Piotr Pokora:
\begin{thm}
Let $\mathcal{L}\subset\mathbb{P}_{\mathbb{C}}^{2}$ be an $n_{2}$-arrangement
with $n\geq3$ and assume that $t_{3}(\cL)=0$. Then $\cL$ is either
the triangle (so $n=3$) or the Hesse arrangement (so that $n=12$).
\end{thm}

\begin{proof}
It is easy to see that if $n=3$, then the only possible arrangement
is the triangle with $t_{2}=3$. We can assume that $n\geq4$. We
can also assume, without loss of generality, that $t_{n}=0$ and $t_{n-1}=0$,
since both trivial and quasi-trivial are not $n_{2}$-arrangements.
Then $\mathcal{L}$ satisfies the Hirzebruch inequality:
\[
t_{2}+t_{3}\geq n+\sum_{r\geq5}(r-4)t_{r}
\]
(see Theorem \ref{thm:(Hirzebruch-inequality-for}). Since $t_{2}=n$
and $t_{3}=0$, we get that for every $r\geq5$ we have $t_{r}=0$.
Summing up, if $\mathcal{L}$ is an $n_{k}$-arrangement with $t_{3}=0$,
then $t_{2}=n$, $t_{4}=\frac{n^{2}-3n}{12}$, and $t_{r}=0$ for
every $r\geq5$. Moreover, our arrangement satisfies the equality
in the above Hirzerbuch's inequality. Our goal is to check which line
arrangements with only double and quadruple points satisfy Hirzebruch's
inequality. It turns out, somehow surprisingly, that it follows from
a result devoted to the existence of ball-quotient surfaces constructed
as Hirzebruch-Kummer covers \cite[Kapitel 3.1 G.]{BHH87}, namely
the only arrangement satisfying all these conditions is the Hesse
arrangement of $12$ lines with $12$ double and $9$ quadruple points. 
\end{proof}

\subsection{\label{subsec:Simplicial-Arrangements-and}Simplicial Arrangements
and $\protect\L$-operators}

An arrangement of real lines is called \textit{simplicial} if all
the polygons cut out by the lines are triangles; this is equivalent
to require the equality 
\begin{equation}
3+\sum_{r\geq2}(r-3)t_{r}=0.\label{eq:Simplicial}
\end{equation}
The quasi-trivial, the polygonal arrangements and extended polygonal
arrangements (see Section \ref{subsec:Classical-arrangements}) are
the three known infinite families of simplicial arrangements. Few
other sporadic examples of simplicial arrangements are known. The
state of the art is given in \cite{CEL} (together with the normals
of the lines of simplicial line arrangements). The simplicial arrangements
are labelled $A(n,k)$ or $A(n,k)_{i}$ where $n$ is the number of
lines, $k,i$ are integers characterizing the arrangement. 

One can check that among the $119$ simplicial arrangements with less
than $37$ lines (which are listed in \cite{CEL}), $100$ are $\L$-constant
for some $\L\in\{\L_{3},\L_{3,4},\L_{3,5},\L_{3,6},\L_{2,3},\L_{2,4},\L_{4}\}$.
For the other arrangements, one has the following relations
\[
\begin{array}{cc}
\L_{2,4}(A(14,112)_{3})=A(15,128)_{2}, & \L_{3,4}(A(24,316))=A(25,336)_{3}\\
\L_{3,4}(A(18,184)_{2})=A(19,200)_{3}, & \L_{3,4}(A(24,320))=A(25,336)_{6}\\
\L_{3,4}(A(19,192)_{1})=A(18,180)_{4}, & \L_{3,5}(A(25,320))=A(24,304)\\
\L_{3,4}(A(20,220)_{3})=A(21,240)_{4}, & \L_{3,5}(A(26,380))=A(31,480)\\
\L_{3,4}(A(20,220)_{4})=A(21,240)_{5}, & \L_{3,5}(A(29,400)=A(31,480)\\
\L_{3,4}(A(23,304))=A(25,336)_{6}, & \L_{3,5}(A(32,544))=A(33,576)\\
\L_{3,6}(A(24,312))=A(25,336)_{5}, & \L_{3,6}(A(36,684))=A(37,720)_{2},
\end{array}
\]
moreover the arrangement $A(26,380)$ is $\L_{\{3,5\},5}$-constant.
For $\cA=A(28,420)_{4}$, one has $\L_{3,\{5\}}\L_{3,5}(\cA)=\cA$,
while $\L_{3,5}(\cA)$ is $\L_{3,5}$-constant, equal to $A(29,448)_{4}$.
Also $A(28,420)_{5}$ and $A(28,420)_{6}$ are permuted under the
operator $\L_{\{3,4,5\},5}$. The images of both arrangements by $\L_{3,5}$
is $A(31,480)$. The operator $\L_{\{3,4,5\},5}$ permutes $A(27,400)$
and $A(29,440)$. 

It is possible to obtain (many) other relations among these arrangements,
such as $\L_{4,5}(A(34,612)_{1})=A(13,96)_{3}$.

\subsection{\label{subsec:Free-arrangements}Free arrangements}

A line arrangement is free if the sheaf of vector fields tangent to
this arrangement splits as a sum of two line bundles over $\PP^{2}$
(see \cite{OT} for the definitions). Roughly speaking, Terao's Conjecture
says that the topology and the geometry of a free arrangement is determined
by its combinatoric. For example the Hesse arrangement is free. A
free arrangement $L_{0}$ is such that the roots of its characteristic
polynomial 
\[
T^{2}+(1-d)T+1-d+\sum_{k\geq2}(k-1)t_{k}(L_{0})
\]
are integers, where $d=|L_{0}|$. 

One can construct free arrangements from a known free arrangement
by the addition-deletion theorem of Terao (see e.g. \cite{Abe}).
A natural question is whether one can construct all free arrangement
starting from the arrangement $\emptyset$ and using the addition-deletion
theorem. An arrangement (in $\PP^{2}$) is said inductively free if
it can be obtained by the addition theorem, recursively free if it
can be obtain by addition and deletion of lines. 

Cuntz-Hoge arrangement (see \cite{CH}) $\text{CH}_{27}$ of $27$
lines was the first found line arrangement which is free but non recursively
free. It is defined over $\QQ(\zeta)$, where $\zeta$ is a primitive
fifth root of unity, this is in fact the reflection arrangement of
the exceptional complex reflection group $G_{27}$. Its singularities
are
\[
t_{2}=15,\,t_{3}=70,\,t_{7}=6.
\]
The line arrangement $\L_{\{7\},\{2\}}(\text{CH}_{27})$ is the icosahedral
line arrangement of $15$ lines with $t_{2}=15,t_{3}=10,t_{5}=6$,
it is also known as the simplicial arrangement $\cA(15,120)$. The
line arrangement $\L_{\{7\},\{2\}}(\text{CH}_{27})$ is contained
in $\text{CH}_{27}$; the complementary arrangement $L_{12}$ of $12$
lines has only double points and $\L_{\{2\},\{6\}}(L_{12})=\cA(15,120)$.

The dual line arrangement $\cD_{\{2\}}(\text{CH}_{27})$ is also (projectively
equivalent to) the icosahedral line arrangement $\cA(15,120)$.

In \cite{ACKN}, Abe, Cuntz, Kawanoue and Nozawa constructed two arrangements
$\text{ACKN}_{13}$ and $\text{ACKN}_{15}$ that are free but not
recursively free arrangements. In Section \ref{sec:First-divergent-arrangements}
we will discuss these arrangements which are interesting for understanding
the dynamic of the operator $\L_{\{3\},\{2\}}$. We found anew these
lines arrangement $\text{ACKN}_{13}$ and $\text{ACKN}_{15}$ in a
completely different way, when studying the operator $\ldt$ (see
Subsections \ref{subsec:A-flashing-6-lines} and \ref{subsec:Povera-arrangements}).
The arrangement $\text{ACKN}_{15}$ is also related to $\cA(15,120)$,
see \cite{RoulleauDy}.

\subsection{\label{subsec:Zariski-pair-arrangements}Zariski pair arrangements}

There are several definitions of Zariski pairs $(L,L')$; roughly
speaking these are lines arrangements that have the same combinatoric
but such that the topology of the pairs $(\PP^{2},L),$ $(\PP^{2},L')$
is different.

$\bullet$ The Zariski pair $\text{Ry}_{13},\text{Ry}_{13}'$ defined
by Rybnikov in \cite{Rybnikov} are $13$ lines arrangements with
$t_{2}=33,t_{3}=15$, they define the same matroid, (meaning the $13$
lines have the same configurations, see Section \ref{sec:Matroids-and-line}).
However one can check that the generic arrangements have different
behavior under the map $\L_{\{3\},\{2\}}$: the arrangement $\ltd(\text{Ry}_{13})$
has $42$ lines but $\ltd(\text{Ry}_{13}')$ has $45$. Also $\L_{3}(\text{Ry}_{13})$
has $11$ lines, but $\L_{3}(\text{Ry}_{13}')$ possess $10$ lines. 

We do not claim that the geometric differences we found between these
line arrangements give another proof that one has Zariski pairs. However,
we remark that in \cite[Proposition 4.6]{GBV}, the different topology
of the pairs is detected thanks to an alignment of three specific
points.

$\bullet$ The Zariski pair $\text{ACCM}_{11},\text{ACCM}_{11}'$
defined by Artal, Carmona, Cogolludo and Marco in \cite{ACCM} are
arrangements of $11$ lines. These line arrangements are defined over
$\QQ(\sqrt{5})$ and are conjugated under the Galois group. One of
the lines has been added in order to trivialize the automorphism group
of the combinatoric and $\L_{3}(\text{ACCM}_{11})$ removes that line:
the line arrangement $\L_{3}(\text{ACCM}_{11})$ is the simplicial
arrangement $\cA(10,60)_{3}$, which is the union of the side of a
pentagon and the lines through the center and the vertices of the
pentagon.

$\bullet$ The Zariski pair $\text{GB}_{12},\text{GB}_{12}'$ defined
by Guerville-Ballé in \cite{GB} have $12$ lines, defined over $\QQ(\ze_{5})$
where $\ze_{5}$ is a primitive $5^{th}$-root of unity. There is
also a line which forces the triviality of the automorphism group
of the combinatoric. Applying $\L_{3}$ removes that line, and then
one gets projectively equivalent line arrangements with $t_{2}=13,t_{3}=6,t_{4}=4$. 

$\bullet$ Various Zariski pairs of $13$ lines over the rationals
are defined by Guerville-Ballé and Viu in \cite{GBV}. One of these
pairs has the following non-bases (see Section \ref{sec:Matroids-and-line}):

\[
\begin{array}{c}
(1,5,7),(1,8,10),(1,11,12),(2,5,6),(2,8,9),(2,11,13),(3,4,5),\\
(3,6,8),(3,7,11),(3,9,10),(3,12,13),(2,4,7,10,12),(1,4,6,9,13).
\end{array}
\]
The moduli space of realizations has two one dimensional irreducible
components $\cM^{\pm}$. The normals of the $13$ lines are the (projectivisation
of) the columns of the following matrix:
\[
\left(\begin{array}{ccccccccccccc}
1 & 0 & 0 & 1 & 1 & 1 & 1 & 1 & 1 & 1 & 1 & 1 & 1\\
0 & 1 & 0 & 1 & 1 & a^{2} & b^{2} & a^{2} & a & a & b^{2} & b & b\\
0 & 0 & 1 & 1 & a^{2} & a^{2} & 1 & a & a & 1 & b & 1 & b
\end{array}\right),
\]
where $ab=\pm1$ for $\cM^{\pm}$, moreover the parameter $a$ must
be in some open set $U^{\pm}$ of $\PP^{1}$ so that the line arrangement
has no more singularities or is not degenerate. Let us denote by
$\cA^{\pm}(a)$ the $13$ lines arrangement associated to the parameter
$a\in U^{\pm}$. Then the line arrangement $\ltd(\cA^{+})$ has $18$
lines, whereas $\ltd(\cA^{-})$ has $30$ lines: the different nature
of $\cA^{\pm}$ is detected by the $\ltd$ operator. %

$\bullet$ The Zariski triple defined by Guerville-Ballé in \cite{GB20}
are arrangements of $12$ lines, defined over $\QQ(\ze_{7})$ where
$\ze_{7}$ is a primitive $7^{th}$-root of unity. Their singularity
set is $t_{2}=17,t_{2}=11,t_{4}=t_{5}=1$. These line arrangements
are rigid and are not projectively equivalent, however they are conjugated
under the Galois group. Therefore, they are indiscernible from the
view-point of line operators since the action of the Galois group
commutes with the action of the line operators. Similar arrangements
are constructed in \cite{GB21}.

\section{\label{sec:First-divergent-arrangements}First divergent arrangements}

\subsection{\label{subsec:Minimal--constant-arrangements}Minimal $\protect\L$-constant
arrangements}

In Section \ref{sec:Some-examples}, we obtained a classification
of line arrangements $L_{0}$ such that $\L_{2}(L_{0})=L_{0}$. The
triangle $T$, i.e. $3$ lines in general position, is the non-trivial
$\L_{2}$-constant configuration which has the minimal number of lines. 

The configuration $L$ of $4$ lines in general position is such that
$\L_{2,3}(L)=L$ and any other arrangement with $4$ lines is sent
to $\emptyset$ by $\L_{2,3}$. The complete quadrilateral is also
$\L_{2,3}$-constant.  One may ask the following question:
\begin{problem}
For a given operator $\L$, find the minimal number $n>0$ of lines
such that there is an arrangement $L$ with $n$ lines and $\L(L)=L$.
\end{problem}

The complete quadrilateral $A_{1}(6)$ is $\L_{3,2}$-constant, and
so are the simplicial arrangements $A(7,32)$ and $A(9,48)$. The
Hesse configuration of $12$ lines is $\L_{3,2}$ and $\L_{4,2}$-constant.
Other $\L_{4,2}$-constant arrangements are $A(13,96)_{3},$ $A(15,120)$;
the first one is defined over $\QQ$.

\subsection{\label{subsec:A-condition-for}A condition for divergence }

We have seen in Example \ref{exa:(The-complete-quadrilateral).} (see
also Section \ref{subsec:Divergence-L_2} below), that four lines
in general position over a field of characteristic $0$ is the configuration
with the minimal number of lines such that the associated $\L_{2}$-sequence
diverges, i.e. the number of lines of the associated $\L_{2}$-sequence
goes to infinity. More generally, for a fixed operator $\L$, it would
be interesting to known which configurations of lines are $\L$-divergent,
and with the minimal number of lines for that property. The following
observation is elementary:
\begin{lem}
\label{lem:BorneLineArrangementdiff}Let $\cA$ be an arrangement
of $m$ lines. Suppose that $\L_{n,k}(\cA)$ contains a line $\ell$
not on $\cA$. Then $m\geq nk$. 
\end{lem}

\begin{proof}
The line goes though at least $k$ points of multiplicity $\geq n$,
since $\ell$ is not a line of $\cA$ there are at least $nk$ lines. 
\end{proof}

\subsection{\label{subsec:Divergence-L_2}Divergence of the $\protect\L_{2}$-operator}

Let $\cC_{0}$ be four real lines in general position and let $(\cC_{n})_{n\in\NN}$
be the associated $\L_{2}$-sequence.
\begin{prop}
The sequence $(\cC_{n})_{n\in\NN}$ is divergent. 
\end{prop}

That follows directly from Melchior inequality (see Theorem \ref{thm:(Melchior)-Let-}).
Indeed, at each step, $\cC_{k}$ is contained in $\cC_{k+1}$ and
no double point in $\cC_{k}$ remains a double point in $\cC_{k+1}$.
The starting line arrangement being real, this is a sequence of real
line arrangements, thus $\cC_{k}$ has some double points, therefore
the number of lines is strictly increasing with $k$. 

We reproduce below another proof from \cite{CW}, which is also interesting. 
\begin{proof}
Since $\cC_{k}\subset\cC_{k+1}$, it is sufficient to prove that $\cC_{k+1}\neq\cC_{k}$
for all $k\geq0$. Suppose that $\cC_{k}=\cC_{k+1}$ for some $k$,
and let $\left\langle \cP_{2}(\cC_{k})\right\rangle $ be the convex
hull of the singular points of $\cC_{k}$. Since $\cC_{0}\subset\cC_{k}$,
$\left\langle \cP_{2}(\cC_{k})\right\rangle $ cannot be a line. If
$\left\langle \cP_{2}(\cC_{k})\right\rangle $ has $4$ or more sides,
then two nonadjacent sides (which are lines in $\cC_{k}$) meet at
some $m$-point of $\cC_{k}$ outside $\left\langle \cP_{2}(\cC_{k})\right\rangle $,
which is impossible. Therefore $\left\langle \cP_{2}(\cC_{k})\right\rangle $
has $3$ sides: let $a,b,c\in\cP_{2}(\cC_{k})$ be the set of vertices
of $\left\langle \cP_{2}(\cC_{k})\right\rangle $. Suppose that there
are points of $\cP_{2}(\cC_{k})$ along at least two sides and not
in $\{a,b,c\}$, say $x\in ab,\,y\in ac$. Then the line $xy\in\cC_{k}$
cuts $bc$ outside of $\left\langle \cP_{2}(\cC_{k})\right\rangle $,
which is again a contradiction. Suppose that there exists some point
$x\in\cP_{2}(\cC_{k})$ in the interior of $\left\langle \cP_{2}(\cC_{k})\right\rangle $,
then the lines $ax,bx,cx$ cut respectively $bc,ac,ab$ at points
of the sides not in $\{a,b,c\}$, which is a contradiction. Therefore
$\cC_{k}$ is a quasi-trivial arrangement, this is impossible since
$\cC_{0}\subset\cC_{k}$, and one concludes that $\cC_{k+1}\neq\cC_{k}$. 
\end{proof}

\subsection{About the divergence of the $\protect\L_{2,3}$-operator\label{subsec:Div-of-Lambda23}}

By Lemma \ref{lem:BorneLineArrangementdiff}, a divergent $\L_{2,3}$-sequence
contains at least $6$ lines. Figure \ref{fig:Div-OpLa-23} represents
the successive images $\cC_{0},\cC_{1},\cC_{2}$ by $\L_{2,3}$ of
the six line arrangement $\cC_{0}$ which is the union of three pairs
of parallel lines (in black). 

\begin{figure}[h]
\begin{center}
\begin{tikzpicture}[scale=0.6]

\clip(-6.02,-4.04) rectangle (5.06,6.02); 
\draw [domain=-6.02:5.06] plot(\x,{(-0.-0.*\x)/6.});
\draw [domain=-6.02:5.06] plot(\x,{(--6.-0.*\x)/6.}); 
\draw  (0.,-4.04) -- (0.,7.02);
\draw  (1.,-4.04) -- (1.,7.02); 
\draw [domain=-6.02:5.06] plot(\x,{(--3.--1.*\x)/1.});
\draw [domain=-6.02:5.06] plot(\x,{(--4.--1.*\x)/1.});
\draw [color=blue,domain=-6.02:5.06] plot(\x,{(-0.--1.*\x)/1.});
\draw [color=blue,domain=-6.02:5.06] plot(\x,{(--4.-0.*\x)/1.}); 
\draw [color=blue] (-3.,-4.04) -- (-3.,7.02); 
\draw [color=green,domain=-6.02:5.06] plot(\x,{(--4.--1.*\x)/2.});
\draw [color=green,domain=-6.02:5.06] plot(\x,{(--3.--2.*\x)/1.}); 
\draw [color=green,domain=-6.02:5.06] plot(\x,{(-1.--1.*\x)/-1.});

\draw[color=blue] (-4.5,2.5) node {$+\ell_\infty$};
\end{tikzpicture}
\end{center} 

\caption{\label{fig:Div-OpLa-23}The line arrangement $\protect\co$ and its
images}
\end{figure}

The arrangement $\cC_{1}=\L_{2,3}(\cC_{0})$ is the union of the black
lines, the 3 blue lines and the line at infinity $\ell_{\infty}$
(which contains $3$ double points of $\cC_{0}$); the arrangement
$\cC_{2}=\L_{2,3}(\cC_{1})$ is the union of $\cC_{1}$ and the three
green lines. The following table gives the number of lines and singularities
of $\cC_{0},\,\cC_{1},\,\cC_{2},\,\cC_{3}$: \\
\begin{tabular}{|c|c|c|c|c|c|c|}
\hline 
 & $|\cC|$ & $t_{2}$ & $t_{3}$ & $t_{4}$ & $t_{6}$ & $t_{7}$\tabularnewline
\hline 
$\cC_{0}$ & $6$ & $15$ &  &  &  & \tabularnewline
\hline 
$\cC_{1}$ & $10$ & $9$ & $6$ & $3$ &  & \tabularnewline
\hline 
$\cC_{2}$ & $13$ & $12$ & $16$ & $3$ &  & \tabularnewline
\hline 
$\cC_{3}$ & $28$ & $87$ & $31$ & $15$ & $3$ & $3$\tabularnewline
\hline 
\end{tabular}
\begin{rem}
The arrangement $\cC_{4}$ has $946$ lines. The arrangement $\cC_{2}'$
obtained by removing the line at infinity of $\cC_{2}$ has $12$
lines and $t_{2}=9,t_{3}=19$. That line arrangement has been obtained
by Zacharias in \cite{Zacharias}. It is studied in \cite{K}, where
it is proved that the $19$ triple points of this configuration give
another rational example of the non-containment of the third symbolic
power into the second ordinary power of an ideal. The $12$ real lines
arrangements with $19$ triple points (the upper bound) have been
classified in \cite{BokPok}.
\end{rem}

\subsection{\label{subsec:DivergenceL32}About the divergence of the $\protect\L_{3,2}$-operator}

Let us give two examples of line arrangements with interesting properties
related to the $\L_{3,2}$-operator.

$\bullet$ The sequence of the number of lines of the successive images
by the operator $\L_{3,2}$ of the dual Hesse configuration begins
by $9,21,57,7401$. 

$\bullet$ The arrangement of $9$ lines $\cC_{0}$ obtained from
the simplicial arrangement $A(10,60)_{1}$ by removing the line at
infinity is such that \\
\begin{tabular}{|c|c|c|c|c|c|c|c|c|c|c|c|c|c|c|c|}
\hline 
 & $|\cC|$ & $t_{2}$ & $t_{3}$ & $t_{4}$ & $t_{5}$ & $t_{6}$ & $t_{7}$ & $t_{8}$ & $t_{9}$ & $t_{10}$ & $t_{11}$ & $t_{12}$ & $t_{16}$ & $t_{18}$ & $t_{20}$\tabularnewline
\hline 
$\cC_{0}$ & $9$ & $6$ & $10$ &  &  &  &  &  &  &  &  &  &  &  & \tabularnewline
\hline 
$\cC_{1}$ & $16$ & $30$ & $4$ & $3$ & $6$ &  &  &  &  &  &  &  &  &  & \tabularnewline
\hline 
$\cC_{2}$ & $25$ & $60$ & $24$ & $3$ &  & $10$ &  &  &  &  &  &  &  &  & \tabularnewline
\hline 
$\cC_{3}$ & $229$ & $8472$ & $1572$ & $468$ & $258$ & $79$ & $42$ & $6$ & $6$ & $3$ & $6$ & $3$ & $12$ & $12$ & $6$\tabularnewline
\hline 
\end{tabular}\\
Here $(\cC_{n})_{n}$ is the $\L_{3,2}$-sequence associated to $\cC_{0}$.
We conjecture that the number of lines of $(\cC_{n})_{n}$ tends to
infinity.

\subsection{\label{subsec:About-L_3}About the divergence of the $\protect\L_{3}$-operator}

The equations of the $15$ lines $\ell_{1},\dots,\ell_{15}$ of the
simplicial arrangement $A(15,132)_{1}$ are given in \cite[Table 7]{CEL},
the singularities of $A(15,132)_{1}$ are 
\[
t_{2}=9,\,t_{3}=22,\,t_{5}=3.
\]
The arrangement $A(15,132)_{1}$ is defined over fields containing
a root of $X^{3}-3X-25$. The configuration $\cC_{0}$ with the fewer
lines we found and which seems divergent for the $\L_{3}$-operator
is the $14$ line arrangement obtained by removing the line $\ell_{2}$
(or $\ell_{10}$ or $\ell_{11}$) in $A(15,132)_{1}$ (each line $\ell_{2},\ell_{10},\ell_{11}$
contains two $5$-points). The singularities of $\cC_{0}$ are 
\[
t_{2}=9,\,t_{3}=20,\,t_{4}=2,\,t_{5}=1.
\]
The following table gives the number of lines and singularities of
the first terms of the associated $\L_{3}$-sequence:\\
\begin{tabular}{|c|c|c|c|c|c|c|c|c|c|c|}
\hline 
 & $|\cC|$ & $t_{2}$ & $t_{3}$ & $t_{4}$ & $t_{5}$ & $t_{6}$ & $t_{7}$ & $t_{8}$ & $t_{9}$ & ...\tabularnewline
\hline 
$\cC_{0}$ & $14$ & $9$ & $20$ & $2$ & $1$ &  &  &  &  & \tabularnewline
\hline 
$\cC_{1}$ & $18$ & $27$ & $12$ & $10$ & $3$ &  &  &  &  & \tabularnewline
\hline 
$\cC_{2}$ & $24$ & $63$ & $12$ & $6$ & $12$ &  & $1$ &  &  & \tabularnewline
\hline 
$\cC_{3}$ & $42$ & $235$ & $46$ & $6$ & $8$ & $8$ & $12$ &  &  & \tabularnewline
\hline 
$\cC_{4}$ & $305$ & $16417$ & $2847$ & $849$ & $314$ & $\text{101}$ & $56$ & $30$ & $16$ & ...\tabularnewline
\hline 
\end{tabular}

The configuration $\cC_{0}$ has another curious feature: the $\L_{\{3,4\},3}$-sequence
$(\tilde{\cC}_{n})_{n\geq0}$ associated to $\cC_{0}$ converges to
a sequence of $14$ lines. The number of lines of the first $8$ terms
of the sequence is $14,18,24,18,23,20,26,14$ so that one has $\tilde{\cC}_{n}\neq\tilde{\cC}_{n+1}$
for $n\leq7$, moreover one has $\tilde{\cC}_{n}=\tilde{\cC}_{n+1}$
for $n\geq8$. 

Let $\cC'_{0}$ be the arrangement obtained from $A(15,132)_{1}$
by removing the line $\ell_{5}$. Forgetting the first term, the associated
$\L_{\{3,4\},3}$-sequence is periodic of period $2$: for $n\geq1$,
one has $\cC'_{n}=\cC'_{n+2}$. The configuration $\cC'_{1}$ has
$13$ lines and singularities $t_{2}=12,\,t_{3}=14,\,t_{4}=4$, the
configuration $\cC'_{2}$ has $17$ lines, and singularities 
\[
t_{2}=32,t_{3}=4,t_{4}=12,t_{5}=3.
\]
The arrangement $\cC_{1}'$ is contained in $\cC_{2}'$. 

\subsection{About the divergence of the $\protect\L_{\{2\},\{2\}}$-operator}

Let $\cA_{0}$ be the union of four lines in general position. Then
$\cA_{1}=\L_{\{2\},\{2\}}(\cA_{0})$ is the union of three lines,
and the associated sequence $(\cA_{k})_{k}$ is constant for $k\geq1$.
Any other configuration of $\leq4$ lines gives a sequence converging
to the empty line arrangement.

If instead one takes the union $\cA_{0}$ of $5$ lines in general
position, then $\cA_{1}=\L_{\{2\},\{2\}}(\cA_{0})$ has $15$ lines,
and $\cA_{2}$ has $2070$ lines. The arrangements $\cA_{0},\cA_{1},\cA_{2}$
are disjoints as sets of lines. We may wonder if that can be generalized:
if for every $\cA_{k},\cA_{k'}$ with $k\neq k'$, the line arrangements
$\cA_{k},\cA_{k'}$ have no common lines. We conjecture that the number
of lines of $\cA_{k}$ tends to $\infty$ with $k$.

\section{\label{Sec:On-the-L=00007B2=00007D=00007B3=00007D-operator}On the
$\protect\L_{\{2\},\{3\}}$-operator}

\subsection{\label{subsec:A-flashing-6-lines}A flashing configuration for $\protect\ldt$.}

Define $\cS=\{0,\pm1,\tfrac{1}{2},2,\tau,\tau^{2}\}$, with $\tau^{2}-\tau+1=0$.
For a fixed parameter $t\notin\cS$, consider the configuration $\cF_{0}=\cF_{0}(t)$
of $6$ lines $\ell_{1},\dots,\ell_{6}$, whose normals are the columns
of the following matrix
\[
\left(\begin{array}{cccccc}
0 & 1 & 1 & 1 & 0 & 1\\
1 & 1 & t & 0 & 0 & t^{2}-t+1\\
0 & 1 & 1 & 0 & 1 & t
\end{array}\right).
\]
The singularities of $\cF_{0}$ are $t_{2}=12,t_{3}=1$. Each line
$\ell_{1},\ell_{2},\ell_{3}$ contains $4$ singular points: the triple
point and $3$ double points of $\cF_{0}$; each of the three remaining
lines contains $5$ double points of $\cF_{0}$. Then $\cF_{1}=\L_{\{2\},\{3\}}(\cF_{0})$
is an arrangement of $6$ lines, union of the $3$ lines $\ell_{1},\ell_{2},\ell_{3}$
of $\cF_{0}$ and three lines $\ell_{7},\ell_{8},\ell_{9}$ with normals
\[
(1:1:0),(1:t:t),(0:t:1).
\]
It has the same configuration as $\cF_{0}$, moreover $\L_{\{2\},\{3\}}(\cF_{1})=\cF_{0}$,
so that the sequence associated to $\cF_{0}$ is periodic with period
$2$. The union of $\cF_{0}$ and $\cF_{1}$ is a $9$ line arrangement
$\cA_{1}$ with $6$ double points and $10$ triple points. That line
arrangement is a non-generic case of a Pappus configuration. 
\begin{defn}
We say that an arrangement projectively equivalent to $\cF_{0}(t)$
for some $t\notin\cS$ is a flashing arrangement of six lines.
\end{defn}

In Figure \ref{fig:The-9-lines} is represented a Flashing arrangement;
in black are the lines $\ell_{1},\ell_{2},\ell_{3}$, in blue the
lines $\ell_{4},\ell_{5},\ell_{6}$ and in red the lines $\ell_{7},\ell_{8},\ell_{9}$;
the triple point is at infinity. 

\begin{figure}[h]
\begin{center}
\begin{tikzpicture}[scale=0.3]

\clip(-5.77,-7.0) rectangle (8.87,5.885);
\draw [line width=0.2mm,domain=-5.77:8.87] plot(\x,{(-0.-0.*\x)/1.});
\draw [line width=0.2mm,domain=-5.77:8.87] plot(\x,{(-1.-0.*\x)/0.5});
\draw [line width=0.2mm,domain=-5.77:8.87] plot(\x,{(-1.-0.*\x)/0.2});
\draw [line width=0.2mm,color=blue] (0.,-8.115) -- (0.,5.885);
\draw [line width=0.2mm,color=blue] (2.,-8.115) -- (2.,5.885);
\draw [line width=0.2mm,color=blue,domain=-5.77:8.87] plot(\x,{(-1.-0.75*\x)/0.95});
\draw [line width=0.2mm,color=red,domain=-5.77:8.87] plot(\x,{(-0.-1.*\x)/1.});
\draw [line width=0.2mm,color=red,domain=-5.77:8.87] plot(\x,{(--2.-1.*\x)/-0.4});
\draw [line width=0.2mm,color=red,domain=-5.77:8.87] plot(\x,{(-1.3333333333333335-1.*\x)/0.6666666666666667});

\end{tikzpicture}
\end{center} 

\caption{\label{fig:The-9-lines}The 9 lines of $\protect\cA_{1}$, for the
flashing arrangement of 6 lines for $\protect\ldt$}
\end{figure}
Using matroids (see Section \ref{sec:Matroids-and-line}), one can
compute that the moduli space of line arrangements with such incidences
is the above family $\cF_{0}=\cF_{0}(t)$, for $t\not\in\cS$. The
parameter $t\in\{0,\pm1\}\,(\subset\cS)$ gives a degenerate line
arrangement. When $t=\tau\in\cS$ with $\tau^{2}-\tau+1=0$, $\cF_{0}(\tau)$
is an arrangement of six lines with two triple points such that the
line arrangement $\ldt(\cF_{0})$ is $\text{Ceva(3)}$, (a.k.a. the
dual Hesse arrangement). When $t\in\{\tfrac{1}{2},2\}$, $\cF_{1}=\ldt(\cF_{0})$
has $7$ lines with four double points and $\ldt(\cF_{1})=\emptyset$. 

For the general case $t\notin\cS$, the projective transformation
of the plane 
\[
\g=\left(\begin{array}{ccc}
-1 & 1 & 1-t\\
-t & t & 1-t\\
-t & 1 & 0
\end{array}\right)\in\text{PGL}_{3}
\]
is such that $\g^{2}=\text{I}_{\text{d}}$ and $\cF_{1}=\g\cF_{0}$;
in particular, the action of $\ldt$ on the moduli space of flashing
configurations is trivial.

We found the first example of a Flashing arrangement when searching
a line arrangement $\co$ with the smallest number of lines such that
the number of lines of the sequence defined by $\cC_{k+1}=\ldt(\cC_{k})$
diverges to $\infty$. 
\begin{rem}
In \cite{NZ}, Nazir and Yoshinaga define two line arrangements $\cA^{\pm}$
of $9$ lines with the same weak combinatoric (same singularities):
$t_{2}=6,t_{3}=10$ as $\cA_{1}=\cF\cup\ldt(\cF)$, for a flashing
arrangement $\cF$. However, they are rigid line arrangements. One
can check moreover that such an arrangement $\cA^{\pm}$ contains
three flashing arrangements $\cF$ of six lines. Each such an arrangement
$\cF$ is such that there exists a unique line in $\ldt(\cF)$ which
is not on $\cA^{\pm}$. The union of $\cA^{\pm}$ and these three
lines is the Ceva$(4)$ arrangement. 
\end{rem}

\subsection{\label{subsec:A-flashing-12-lines}A flashing configuration of $12$
lines for $\protect\L_{\{3\},\{2\}}$}

Let $\cF_{0}$ be a flashing arrangement of six lines and define the
dual of $\cF_{0}$ by $\cA_{0}=\cD_{\{2\}}(\cF_{0})$ (see notations
in Section \ref{fig:The-dual-configuration}).  This is a $12$ lines
arrangement, with $t_{2}=18,t_{3}=6,t_{5}=3.$ The arrangement $\cA_{1}=\L_{\{3\},\{2\}}(\cA_{0})$
has the same number and type of singularities; one has $\cA_{0}=\L_{\{3\},\{2\}}(\cA_{1})$:
the $\L_{\{3\},\{2\}}$-sequence associated to $\cA_{0}$ is periodic
with period $2$. 
\begin{defn}
We call $\cA_{0}$ a flashing configuration of $12$ lines.
\end{defn}

Arrangements $\cA_{0}$ and $\cA_{1}$ have nine common lines and
the union of $\cA_{0}$ and $\cA_{1}$ is a $15$ lines arrangement
with singularities
\[
t_{2}=36,t_{3}=3,t_{5}=6.
\]
One can compute that the moduli space of line arrangements of $12$
lines defining the same matroid as $\cA_{0}$ is two dimensional.
However for a generic element $\cB$ of that moduli space, $\Lambda_{\{3\},\{2\}}(\cB)$
has $15$ lines with singularities $t_{2}=36,t_{3}=3,t_{5}=6,$ therefore
$\Lambda_{\{3\},\{2\}}$ is not a self-map of that moduli space. The
difference is also seen by looking at the $9$ lines arrangement $\cD(\cP_{3}(\cB))$
which has $t_{2}=t_{3}=9$, whereas the arrangement $\cD(\cP_{3}(\cA_{0}))$
is a $9$ line arrangement with $t_{2}=6,t_{3}=10$.

The arrangement $\check{\cF_{0}}=\cD_{2}(\cF_{0})$ has $13$ lines
and singularities 
\[
t_{2}=21,\,t_{3}=t_{4}=t_{5}=3.
\]
Its moduli space $\cM_{13}$ is one dimensional, and $\L_{\{3\},2}(\cA_{1})=\check{\cF_{0}}$,
so that we may recover the moduli space of flashing configuration
of $12$ lines from the moduli space $\cM_{13}$. The difference between
$\check{\cF_{0}}$ and $\cA_{1}$ is the line going through points
$\ell_{1},\ell_{2},\ell_{3}$ in Figure \ref{fig:The-13-lines}. The
black points correspond to the common lines $\ell_{1},\ell_{2},\ell_{3}$
of $\cF_{0}$ and $\check{\cF_{1}}=\cD_{2}(\cF_{1})$, the blue points
correspond to $\ell_{4},\ell_{5},\ell_{6}$ , the red points correspond
to $\ell_{7},\ell_{8},\ell_{9}$ (the last being at infinity), the
three blue (respectively red) lines are the lines of $\check{\cF_{0}}$
(resp. $\check{\cF_{1}}$) which are not in the arrangement $\check{\cF_{1}}$
(resp. $\check{\cF_{0}}$), the $10$ black lines are common to both
$\check{\cF_{0}}$ and $\check{\cF_{1}}$.

\begin{figure}[h]
\begin{center}
\begin{tikzpicture}[scale=2.5]

\clip(1.0276543209876348,0.5969547325102883) rectangle (3.548148148148086,3.0226337448559666);
\draw [domain=1.0276543209876348:3.548148148148086] plot(\x,{(--1.2360679700000006-0.*\x)/0.6180339850000003});
\draw [domain=1.0276543209876348:3.548148148148086] plot(\x,{(--3.2360679700000006-1.*\x)/0.6180339850000003});
\draw [line width=0.2mm] (2.,0.5969547325102883) -- (2.,3.0226337448559666);
\draw [domain=1.0276543209876348:3.548148148148086] plot(\x,{(-0.-0.38196601499999994*\x)/-0.38196601499999994});
\draw [domain=1.0276543209876348:3.548148148148086] plot(\x,{(--0.23606797838501947-0.38196601499999994*\x)/-0.38196601500000016});
\draw [color=blue,domain=1.0276543209876348:3.548148148148086] plot(\x,{(-0.9999999916149807-0.38196601499999994*\x)/-1.});
\draw [line width=0.2mm] (2.6180339850000003,0.5969547325102883) -- (2.6180339850000003,3.0226337448559666);
\draw [color=blue,domain=1.0276543209876348:3.548148148148086] plot(\x,{(--1.618033985-0.6180339850000001*\x)/0.38196601499999994});
\draw [domain=1.0276543209876348:3.548148148148086] plot(\x,{(--2.6180339766149805-0.6180339850000001*\x)/1.});
\draw [domain=1.0276543209876348:3.548148148148086] plot(\x,{(-0.6180339850000003-0.*\x)/-0.6180339850000003});
\draw [domain=1.0276543209876348:3.548148148148086] plot(\x,{(-1.--0.6180339850000001*\x)/0.23606797000000013});
\draw [domain=1.0276543209876348:3.548148148148086] plot(\x,{(-0.9999999916149802-0.*\x)/-0.6180339850000001});
\draw [color=blue,domain=1.0276543209876348:3.548148148148086] plot(\x,{(--1.3819660149999997-1.*\x)/-0.6180339850000003});
\draw [color=red,domain=1.0276543209876348:3.548148148148086] plot(\x,{(--1.6180339766149805-1.2360679700000003*\x)/-0.6180339850000003});
\draw [color=red,domain=1.0276543209876348:3.548148148148086] plot(\x,{(--4.660100493300001--0.001966014999999821*\x)/1.781966015});
\draw [color=red,domain=1.0276543209876348:3.548148148148086] plot(\x,{(--3.320650466-0.001966015000000043*\x)/2.4});

\begin{scriptsize}
\draw [fill=black] (2.23606797,1.618033985) circle (0.2mm);
\draw[color=black] (2.1341975308641573,1.6813991769547325) node {$\ell_1$};

\draw [fill=blue] (2.6180339850000003,2.) circle (0.2mm);
\draw[color=blue] (2.5474,2.0646) node {$\ell_4$};

\draw [fill=blue] (1.618033985,1.618033985) circle (0.2mm);
\draw[color=blue] (1.6439,1.7413) node {$\ell_5$};

\draw [fill=black] (2.,2.) circle (0.2mm);
\draw[color=black] (1.9,2.07) node {$\ell_3$};

\draw [fill=black] (2.6180339850000003,1.) circle (0.2mm);
\draw[color=black] (2.7074074074073606,1.0651028806584362) node {$\ell_2$};

\draw [fill=blue] (2.,1.) circle (0.2mm);
\draw[color=blue] (1.897,1.073) node {$\ell_6$};

\draw [fill=red] (2.6180339850000003,2.6180339850000003) circle (0.2mm);
\draw[color=red] (2.54740,2.68485) node {$\ell_7$};

\draw [fill=red] (2.,1.381966015) circle (0.2mm);
\draw[color=red] (1.9471,1.48831) node {$\ell_8$};

\draw[color=red] (3.4,1.5) node {$\to \ell_9$};
\end{scriptsize}

\end{tikzpicture}
\end{center} 

\caption{\label{fig:The-13-lines}The 13 lines of $\check{\protect\cF}_{0}=\protect\cL_{2}(\protect\cD(\protect\cC_{0}))$}
\end{figure}

The arrangements of $13$ lines $\check{\cF}_{0}$ has been studied
in \cite{ACKN} as an example of a free but not recursively free line
arrangement (see Section \ref{subsec:Free-arrangements}). 

\subsection{\label{subsec:Povera-arrangements} Unassuming arrangements}

In \cite{RoulleauDy}, we study lines arrangements $\co$ of six lines
such that $t_{2}(\co)=15$, but the line arrangement $\cD_{2}(\co)$
has singularities $t_{2}=27,\,t_{3}=t_{5}=6$, moreover the six points
in the dual $\cD(\co)$ are not contained in a conic. We call such
arrangements ``unassuming'': which word describes something that
is deceptively simple but has hidden qualities or advantages. 

There is a one dimensional family of unassuming arrangements; the
normals of the six lines are given by the columns of the following
matrix
\begin{equation}
M_{t}=\left(\begin{array}{cccccc}
1 & 0 & 0 & 1 & \tfrac{1}{2}(1+t) & \tfrac{1}{2}(1-t)\\
0 & 1 & 0 & 1 & \tfrac{1}{2}(1-t) & \tfrac{1}{2}(1+t)\\
0 & 0 & 1 & 1 & 1 & 1
\end{array}\right),\label{eq:MatPov}
\end{equation}
for $t\in U=\PP^{1}\setminus\{0,\pm1,\infty,\pm2\pm\sqrt{5}\}$. For
fixed $t\in U$, let $(\cC_{k})_{k\geq0}$ be the $\ldt$-sequence
associated to $\co=\co(t)$. We obtain that
\begin{thm}
(\cite{RoulleauDy}). \label{thm:The-moduli-space}The moduli space
of unassuming arrangements is the union of $U$ and a point. The image
by the operator $\ldt$ of a generic unassuming arrangement is again
an unassuming arrangement. For the general unassuming arrangement,
the associated $\ldt$-sequence $(\cC_{k})_{k\geq0}$ is such that
$\cC_{m}$ is not projectively equivalent to $\cC_{n}$ for any $m\neq n$.
For each $n$, there exists periodic unassuming arrangements with
period $n$.
\end{thm}

See Figure \ref{fig:The-povera-arr} for a picture of such line arrangement
(in black) and its image (in blue) by $\ldt$.

\begin{figure}[h]
\begin{center}

\begin{tikzpicture}[scale=0.6]

\clip(-8.06,-3.52) rectangle (8.3,6.04);
\draw [line width=0.3mm,domain=-8.06:8.3] plot(\x,{(-0.-0.4*\x)/1.});
\draw [line width=0.3mm] (1.,-3.52) -- (1.,6.04);
\draw [line width=0.3mm] (-1.,-3.52) -- (-1.,6.04);
\draw [line width=0.3mm,domain=-8.06:8.3] plot(\x,{(-0.--0.4*\x)/1.});
\draw [line width=0.3mm,domain=-8.06:8.3] plot(\x,{(-1.-0.*\x)/1.});
\draw [line width=0.3mm,domain=-8.06:8.3] plot(\x,{(-1.-0.*\x)/-1.});
\draw [line width=0.3mm,color=blue,domain=-8.06:8.3] plot(\x,{(-0.-1.*\x)/-1.});
\draw [line width=0.3mm,color=blue,domain=-8.06:8.3] plot(\x,{(-0.-1.*\x)/1.});
\draw [line width=0.3mm,color=blue,domain=-8.06:8.3] plot(\x,{(-0.4-0.*\x)/1.});
\draw [line width=0.3mm,color=blue] (-2.5,-3.52) -- (-2.5,6.04);
\draw [line width=0.3mm,color=blue,domain=-8.06:8.3] plot(\x,{(--0.4-0.*\x)/1.});
\draw [line width=0.3mm,color=blue] (2.5,-3.52) -- (2.5,6.04);
\begin{scriptsize}
\draw[color=black] (-7.66,3.41) node {$\ell_6$};
\draw[color=black] (1.3,4.73) node {$\ell_3$};
\draw[color=black] (-0.6,4.73) node {$\ell_4$};
\draw[color=black] (-7.66,-2.71) node {$\ell_5$};
\draw[color=black] (-7.66,-0.75) node {$\ell_2$};
\draw[color=black] (-7.66,1.33) node {$\ell_1$};
\draw[color=blue] (-3.75,-3.19) node {$\ell_5'$};
\draw[color=blue] (-5.2,5.73) node {$\ell_6'$};
\draw[color=blue] (-7.6,-0.05) node {$\ell_2'$};
\draw[color=blue] (-2.06,5.73) node {$\ell_4'$};
\draw[color=blue] (-7.66,0.69) node {$\ell_1'$};
\draw[color=blue] (2.86,5.73) node {$\ell_3'$};

\draw [fill=black] (-1.,1.) circle (0.75mm);
\draw[color=black] (-0.7,1.3) node {$p_1$};

\draw [fill=black] (1.,1.) circle (0.75mm);
\draw[color=black] (0.668,1.3) node {$p_2$};

\draw [fill=black] (1.,-1.) circle (0.75mm);
\draw[color=black] (0.7,-1.25) node {$p_3$};

\draw [fill=black] (-1.,-1.) circle (0.75mm);
\draw[color=black] (-0.65,-1.25) node {$p_4$};

\draw [fill=black] (0.,0.) circle (0.75mm);
\draw[color=black] (0.,0.33) node {$q_1$};

\draw [fill=black] (2.5,1.) circle (0.75mm);
\draw[color=black] (2.25,1.3) node {$p_5$};

\draw [fill=black] (2.50,-1.) circle (0.75mm);
\draw[color=black] (2.25,-1.25) node {$p_6$};

\draw [fill=black] (-2.4448275862068964,0.9779310344827586) circle (0.75mm);
\draw[color=black] (-2.2,1.3) node {$p_7$};

\draw [fill=black] (-2.5,-1.) circle (0.75mm);
\draw[color=black] (-2.15,-1.25) node {$p_8$};

\draw [fill=black] (1.,0.4) circle (0.75mm);
\draw[color=black] (1.3,0.73) node {$p_9$};

\draw [fill=black] (1.,-0.4) circle (0.75mm);
\draw[color=black] (1.46,-0.17) node {$p_{10}$};

\draw [fill=black] (-0.9689655172413792,0.3875862068965517) circle (0.75mm);
\draw[color=black] (-0.65,0.65) node {$p_{11}$};

\draw [fill=black] (-1.,-0.4) circle (0.75mm);
\draw[color=black] (-1.44,-0.17) node {$p_{12}$};

\draw [fill=black] (0.,5.74) circle (0.75mm);
\draw[color=black] (0.22,5.31) node {$\uparrow q_3$};

\draw [fill=black] (7.7,0.) circle (0.75mm);
\draw[color=black] (7,0) node {$\to q_2$};
\end{scriptsize}

\end{tikzpicture}

\end{center} 

\caption{\label{fig:The-povera-arr}An unassuming arrangement and its image
by $\protect\ldt$}
\end{figure}

For an analytic proof, see \cite{RoulleauDy}. From Figure \ref{fig:The-povera-arr},
let us give a synthetic proof that there exists a one parameter family
of real unassuming line arrangements which is preserved by $\ldt$,
as follows:
\begin{proof}
Let $(\ell_{1},\ell_{2})$ (respectively $(\ell_{3},\ell_{4})$) be
a pair of parallel lines such that $\ell_{1}$ is orthogonal to $\ell_{3}$.
Let $p_{1},p_{2},p_{3},p_{4}$ be the meeting points in $\RR^{2}\subset\PP^{2}(\RR)$
of $\ell_{1},\dots,\ell_{4}$ and let $\ell_{5}$ be a generic line
passing through the intersection point $q_{1}$ of the two diagonals
of the $4$-gon $p_{1},\dots,p_{4}$ (there is a one parameter choice
for such $\ell_{5}$). Let $\ell_{6}$ be the image of $\ell_{5}$
by the orthogonal reflexion $\s_{1}$ with axis the line passing through
$q_{1}$ and parallel to $\ell_{1}$. The line $\ell_{6}$ is also
the image of $\ell_{5}$ by the orthogonal reflexion $\s_{2}$ with
axis the line passing through $q_{1}$ and parallel to $\ell_{3}$.
The lines $\ell_{1},\dots,\ell_{6}$ are the black lines in Figure
\ref{fig:The-povera-arr}. The lines $\ell_{1},\ell_{2}$ (respectively
$\ell_{3},\ell_{4}$) meet at infinity at point $q_{2}$ (respectively
$q_{3}$). We denote by $p_{5},\dots,p_{12}$ the remaining double
points as in Figure \ref{fig:The-povera-arr}. 

Since $p_{11}=\s_{2}(p_{9})$, the line $\ell_{1}'=\overline{p_{9},p_{11}}$
is parallel to $\ell_{1}$ and $\ell_{2}$ and therefore contains
$q_{2}$. Since $p_{12}=\s_{2}(p_{10})$, the line $\ell_{2}'=\overline{p_{10},p_{12}}$
is also parallel to $\ell_{1}$ and $\ell_{2}$, thus contains $q_{2}$.\\
Since $p_{6}=\s_{1}(p_{5})$, the line $\ell_{3}'=\overline{p_{5},p_{6}}$
is parallel to $\ell_{3}$, $\ell_{4}$ and contains $q_{3}$.\\
Since $p_{8}=\s_{1}(p_{7})$, the line $\ell_{4}'=\overline{p_{7},p_{8}}$
is parallel to $\ell_{3},\ell_{4}$ and contains $q_{3}$.\\
The lines $\ell_{5}'=\overline{p_{2},p_{4}}$ and $\ell_{6}'=\overline{p_{1},p_{3}}$
contain $q_{1}$.\\
The lines $\ell_{1}',\dots,\ell_{6}'$ containing exactly three double
points of $\ell_{1},\dots,\ell_{6}$ are the lines in blue in \ref{fig:The-povera-arr}. 

Then the situation for the blue of lines is the same as for the black
lines, therefore the configuration repeats itself. 
\end{proof}
By duality, one can rephrase the results in Theorem \ref{thm:The-moduli-space}
on the action of $\L_{\{2\},\{3\}}$ on unassuming arrangements as
follows:
\begin{thm}
\label{thm:Main3}For a set $P_{6}=\{p_{1},\dots,p_{6}\}$ of six
points, consider the following property:
\begin{lyxlist}{00.00.0000}
\item [{$(P)$}] The union of the lines containing two points in $P_{6}$
possesses exactly six triple points $p_{1}',\dots,p_{6}'$. The points
of $P_{6}$ are not inscribed in a conic.
\end{lyxlist}
Suppose that $P_{6}$ satisfies $(P)$. Then the set of triple points
$P_{6}'=\{p_{1}',\dots,p_{6}'\}=\Psi_{\{2\},\{3\}}(P_{6})$ satisfies
$(P)$, moreover if the points in $P_{6}$ are real, there exists
a unique set of six real points $P_{6}^{-}$ satisfying (P) and such
that $\Psi_{\{2\},\{3\}}(P_{6}^{-})=P_{6}$.
\end{thm}

Figure \ref{fig:The-dual-configuration} gives an example of such
a set $P_{6}$ (points in black), and its image $P_{6}'$ (points
in red) by $\Psi_{\{2\},\{3\}}$.
\begin{figure}[h]
\begin{center}

\begin{tikzpicture}[scale=2.5]

\clip(0.8386887503215222,0.4097461379762637) rectangle (3.181712280433072,2.848605609375874);
\draw [domain=0.8386887503215222:3.181712280433072] plot(\x,{(-1.2360680732947948--2.*\x)/1.2360679});
\draw [line width=0.2mm] (1.8944271599999993,0.4097461379762637) -- (1.8944271599999993,2.848605609375874);
\draw [domain=0.8386887503215222:3.181712280433072] plot(\x,{(--1.788854285341042-0.34164074000000033*\x)/0.5527864199999999});
\draw [domain=0.8386887503215222:3.181712280433072] plot(\x,{(--2.3105973134818214E-8-0.1842621400000004*\x)/-0.29814238666666704});
\draw [domain=0.8386887503215222:3.181712280433072] plot(\x,{(-0.3685242684470138--0.2981423866666667*\x)/0.11388024666666663});
\draw [domain=0.8386887503215222:3.181712280433072] plot(\x,{(--1.541640726136417-0.*\x)/0.89442716});
\draw [domain=0.8386887503215222:3.181712280433072] plot(\x,{(--2.5527863506820836-0.34164074000000033*\x)/1.10557284});
\draw [domain=0.8386887503215222:3.181712280433072] plot(\x,{(-0.21114571465895904--1.1055728400000002*\x)/1.44721358});
\draw [domain=0.8386887503215222:3.181712280433072] plot(\x,{(--0.4721358000000002--0.7639320999999999*\x)/2.});
\draw [domain=0.8386887503215222:3.181712280433072] plot(\x,{(--2.341640774658959-1.44721358*\x)/-0.3416407400000001});
\draw [domain=0.8386887503215222:3.181712280433072] plot(\x,{(--3.0786892422350673-1.0786893000000002*\x)/0.25464403333333285});
\draw [domain=0.8386887503215222:3.181712280433072] plot(\x,{(-1.3167183813641654--0.34164074000000033*\x)/-0.21114568});
\draw [domain=0.8386887503215222:3.181712280433072] plot(\x,{(-1.0786892999999997-0.2546440333333333*\x)/-1.3333333333333333});
\draw [domain=0.8386887503215222:3.181712280433072] plot(\x,{(-0.6666667244315985--1.3333333333333333*\x)/1.0786892999999997});
\draw [domain=0.8386887503215222:3.181712280433072] plot(\x,{(-0.40000000693179194--0.5527864200000003*\x)/0.5527864200000001});
\begin{scriptsize}
\draw [fill=black] (2.44721358,1.7236067900000003) circle (0.2mm);
\draw [fill=black] (2.2360679,2.61803395) circle (0.2mm);

\draw [fill=black] (1.89442716,1.7236067900000003) circle (0.2mm);
\draw [fill=black] (3.,1.38196605) circle (0.2mm);
\draw [fill=black] (1.,0.61803395) circle (0.2mm);
\draw [fill=black] (1.89442716,1.17082037) circle (0.2mm);
\draw [fill=red] (2.0786892999999997,1.9513672833333333) circle (0.2mm);
\draw [fill=red] (2.490711933333333,1.5393446499999999) circle (0.2mm);
\draw [fill=red] (1.5527864200000001,0.82917963) circle (0.2mm);
\draw [fill=red] (1.6666666666666667,1.1273220166666666) circle (0.2mm);
\draw [fill=red] (1.89442716,2.06524753) circle (0.2mm);
\draw [fill=red] (2.78885432,1.7236067900000003) circle (0.2mm);
\end{scriptsize}

\end{tikzpicture}

\end{center} 

\caption{\label{fig:The-dual-configuration}The points in $P_{6}$ and the
six associated triple points}
\end{figure}

\begin{rem}
Let $\co$ be an unassuming arrangement, then the line arrangement
$\cD_{2}(\co)$ is the line arrangement we denoted by $\text{ACKN}_{15}$
in Section \ref{subsec:Free-arrangements}, which is a free but non
recursively free arrangement. 
\end{rem}

\begin{rem}
\label{rem:Cremona}One may wonder if there exists a Cremona transformation
$\tau$ sending the six lines of an unassuming arrangement $\cC_{0}$
to $\cC_{1}=\ldt(\co)$. We checked that the images of $\co$ by the
Cremona standard involutions based at three of the double points of
$\co$ give either arrangements with less lines or unassuming arrangements
projectively equivalent to $\co$. By the Noether-Castelnuovo theorem,
the Cremona group is generated by standard involution and $PGL_{3}(\CC)$:
the existence of such a $\tau$ seems therefore unlikely.%
\end{rem}

\begin{rem}
One can compute that the set $\cS$ of irreducible conics containing
at least six of the $15$ double points $\cP_{\{2\}}(\co)$ of $\cC_{0}$
is the union of $12$ conics. Each conic $C$ of $\cS$ contains exactly
$6$ double points and the six lines of $\co$ form an hexagon for
$C$. The Pascal line of that hexagon i.e. the line containing the
three double points of the three pairs of opposite sides, is a line
of $\cC_{1}=\ldt(\cC_{0})$. Each line of $\cC_{1}$ is the Pascal
line of two conics of $\cS$, see Figure \ref{fig:The-dual-configuration-1}.
\end{rem}

\begin{figure}[h]
\begin{center}

\begin{tikzpicture}[scale=0.75]

\clip(-3.542962962962968,-2.029629629629638) rectangle (5.6837037037037135,2.0237037037037043);
\draw [line width=0.3mm,domain=-3.542962962962968:5.6837037037037135] plot(\x,{(-0.-0.4*\x)/1.});
\draw [line width=0.3mm] (1.,-2.829629629629638) -- (1.,2.8237037037037043);
\draw [line width=0.3mm] (-1.,-2.829629629629638) -- (-1.,2.8237037037037043);
\draw [line width=0.3mm,domain=-3.542962962962968:5.6837037037037135] plot(\x,{(-0.--0.4*\x)/1.});
\draw [line width=0.3mm,domain=-3.542962962962968:5.6837037037037135] plot(\x,{(-1.-0.*\x)/1.});
\draw [line width=0.3mm,domain=-3.542962962962968:5.6837037037037135] plot(\x,{(-1.-0.*\x)/-1.});
\draw [line width=0.3mm,color=blue] (-2.5,-2.829629629629638) -- (-2.5,2.8237037037037043);
\draw [rotate around={0.:(1.75,0.)},line width=0.3mm,color=green] (1.75,0.) ellipse (2.9825883613622133cm and 1.0331989159885913cm);
\draw [samples=50,domain=-0.99:0.99,rotate around={90.:(0.75,0.)},xshift=0.75cm,yshift=0.cm,line width=0.3mm,color=green] plot ({0.3774917217635375*(1+(\x)^2)/(1-(\x)^2)},{0.7133922984085065*2*(\x)/(1-(\x)^2)});
\draw [samples=50,domain=-0.99:0.99,rotate around={90.:(0.75,0.)},xshift=0.75cm,yshift=0.cm,line width=0.3mm,color=green] plot ({0.3774917217635375*(-1-(\x)^2)/(1-(\x)^2)},{0.7133922984085065*(-2)*(\x)/(1-(\x)^2)});
\end{tikzpicture}

\end{center} 

\caption{\label{fig:The-dual-configuration-1}The two conics with the same
Pascal line for $\protect\co$}
\end{figure}

\subsection{About the divergence of $\protect\ldt$}

Let $\cC_{0}$ be the arrangement of six lines with normals that are
columns of matrix $M_{t}$ in \eqref{eq:MatPov} for $t=\pm2\pm\sqrt{5}\in\PP^{1}\setminus U$.
The following table gives the number of lines and singularities of
the first terms of the associated $\ldt$-sequence:

\begin{tabular}{|c|c|c|c|c|c|}
\hline 
 & $|\cC|$ & $t_{2}$ & $t_{3}$ & $t_{4}$ & $t_{5}$\tabularnewline
\hline 
$\cC_{0}$ & $6$ & $15$ &  &  & \tabularnewline
\hline 
$\cC_{1}$ & $10$ & $45$ &  &  & \tabularnewline
\hline 
$\cC_{2}$ & $90$ & $1710$ & $120$ & $210$ & $45$\tabularnewline
\hline 
\end{tabular}

We conjecture that the number of lines of $\cC_{k}$ tends to infinity
with $k$. The line arrangement $\cD_{\{2\}}(\co)$ is the simplicial
line arrangement $A(15,120)$.

\section{\label{sec:Flashing-arrangements-for}Flashing arrangements for $\protect\L_{\{2\},\{k\}}$,
$k\in\{4,5,6,7\}$}

\subsection{A generalization of the $\protect\ldt$- flashing arrangements}

For $n\geq4$, let us construct line arrangements that have the same
properties for $\L_{\{2\},\{n\}}$ as the flashing arrangement for
$\ldt$ in Section \ref{Sec:On-the-L=00007B2=00007D=00007B3=00007D-operator}.
We thus want to construct a line arrangement $L$ of $3n$ lines $\ell_{1},\dots,\ell_{3n}$
such that the line arrangement $L_{2}=\{\ell_{n+1},\dots,\ell_{2n}\}$
has a unique singular point ($t_{n}(L_{2})=1$), the lines arrangements
$L_{1}=\{\ell_{1},\dots,\ell_{n}\}$ and $L_{3}=\{\ell_{2n+1},\dots,\ell_{3n}\}$
have only nodal singularities ($t_{2}(L_{k})=\tfrac{1}{2}n(n-1)$
for $k=1,3$), the line arrangement $\cC_{0}=L_{1}\cup L_{2}$ and
$\cC_{1}=L_{2}\cup L_{3}$ have nodes and one $n$-point ($t_{2}(\cC_{j})=\tfrac{1}{2}(3n^{2}-n),\,t_{n}(\cC_{j})=1$
for $j=1,2$), moreover the line arrangement $L=L_{1}\cup L_{2}\cup L_{3}$
has singularities 
\[
t_{2}=n^{2}-n,\,t_{3}=n^{2},\,t_{n}=1.
\]
We require moreover that each line $\ell_{k}$ contains $n$ triple
points of $L$. Then, up to permutation of the lines in each of the
arrangements $L_{1},L_{2},L_{3}$ and up to relabeling the triple
points, one may suppose that the incidence matrix of the $3n$ lines
and $n^{2}$ triple points is 
\[
\left(\begin{array}{ccc}
C_{1} & I_{n} & I_{n}\\
C_{2} & I_{n} & G_{1}\\
\vdots & \vdots & \vdots\\
C_{n} & I_{n} & G_{n-1}
\end{array}\right)
\]
where $C_{k}$ is the $n\times n$ matrix having $1$ in all the entries
of $k$-th column and $0$ in all other entries, $I_{n}$ is the size
$n$ identity matrix and $G_{1},\dots,G_{n}$ are $n\times n$ matrices.
Since the lines of that incidence matrix represent the triple points,
each line of these matrices $G_{k}$ must contain a unique $1$, moreover
since the lines are distinct, the matrix $G_{1}$ cannot have a $1$
on its diagonal. Similarly, once $G_{1}$ is fixed, there are new
restrictions on $G_{2}$ etc, so that the matrix 
\[
I_{n}+\sum_{k=1}^{n-1}G_{k}
\]
 is the matrix with 1 in each entries. We complete the above incidence
matrix by adding as a line the size $1\times3n$ matrix $(0,\dots,0,1\dots,1,0\dots,0)$,
which represents the $n$-point.

Let $G$ be the size $n\times n$ matrix such that $G(e_{i})=e_{i+1}$,
where the indices are taken modulo $n$ and $(e_{i})_{i}$ is the
canonical basis. The matrix $G$ has order $n$; for $G_{1},\dots,G_{n-1}$,
let us take $G_{k}=G^{k}$. We checked that for $n\in\{3,4,5,6,7\}$,
the matroid associated to the above incidence matrix is representable
in characteristic $0$, in other words, there exist $3n$ lines in
$\PP^{2}(\CC)$ that have that configuration, see Figure \ref{fig:The-18-lines-flash}
for the case $n=6$. 

In each case, the moduli space of such lines arrangement is one dimensional.
If $n\in\{3,4,6,7\}$, that moduli space is an open subset of $\PP^{1}$,
in particular it is irreducible; if $n=5$, the moduli space is the
union of two disjoint irreducible rational components defined over
$\QQ(\sqrt{5})$. For any $n\in\{3,4,5,6,7\}$, these line arrangements
are made so that 
\[
\L_{\{2\},\{n\}}(\cC_{j})=\cC_{k},
\]
where $\{j,k\}=\{0,1\}$, indeed: the $n^{2}$ double points which
are the intersection points of the lines in $L_{1}$ and the lines
in $L_{2}$ are also the $n^{2}$ double points which are the intersection
of the lines in $L_{2}$ and the lines in $L_{3}$. The lines that
are common to $\cC_{0}$ and $\cC_{1}$ are the lines of $L_{2}$.
Each line in $L_{1}$ possess $2n-1$ double points in $\cC_{0}$,
thus they cannot be in $\cC_{1}$, whereas $L_{3}\subset\L_{\{2\},\{n\}}(\cC_{0})$.
In the next section, we describe with more details the case $k=4$.

\begin{figure}[h]
\begin{center}
\begin{tikzpicture}[scale=0.25]

\clip(-15.611921156381017,-8.671198543004897) rectangle (15.620855067047883,16.50871407123237);
\draw [line width=0.2mm,color=blue] (0.,-8.671198543004897) -- (0.,16.50871407123237);
\draw [line width=0.2mm,color=blue] (1.5,-8.671198543004897) -- (1.5,16.50871407123237);
\draw [line width=0.2mm,color=blue,domain=-15.611921156381017:15.620855067047883] plot(\x,{(-0.--2.*\x)/1.});
\draw [line width=0.2mm,color=blue,domain=-15.611921156381017:15.620855067047883] plot(\x,{(--3.--1.*\x)/1.});
\draw [line width=0.2mm,color=blue,domain=-15.611921156381017:15.620855067047883] plot(\x,{(--4.5--1.*\x)/1.});
\draw [line width=0.2mm,color=blue,domain=-15.611921156381017:15.620855067047883] plot(\x,{(--6.--2.*\x)/1.});
\draw [line width=0.2mm,domain=-15.611921156381017:15.620855067047883] plot(\x,{(--5.142857142857142-0.*\x)/1.});
\draw [line width=0.2mm,domain=-15.611921156381017:15.620855067047883] plot(\x,{(--7.-0.*\x)/1.});
\draw [line width=0.2mm,domain=-15.611921156381017:15.620855067047883] plot(\x,{(--13.5-0.*\x)/1.});
\draw [line width=0.2mm,domain=-15.611921156381017:15.620855067047883] plot(\x,{(-6.-0.*\x)/1.});
\draw [line width=0.2mm,domain=-15.611921156381017:15.620855067047883] plot(\x,{(--1.8-0.*\x)/1.});
\draw [line width=0.2mm,domain=-15.611921156381017:15.620855067047883] plot(\x,{(--3.75-0.*\x)/1.});
\draw [line width=0.2mm,color=red,domain=-15.611921156381017:15.620855067047883] plot(\x,{(--5.142857142857142--1.2380952380952381*\x)/1.});
\draw [line width=0.2mm,color=red,domain=-15.611921156381017:15.620855067047883] plot(\x,{(--7.--4.333333333333331*\x)/1.});
\draw [line width=0.2mm,color=red,domain=-15.611921156381017:15.620855067047883] plot(\x,{(--13.5-13.*\x)/1.});
\draw [line width=0.2mm,color=red,domain=-15.611921156381017:15.620855067047883] plot(\x,{(-6.--5.2*\x)/1.});
\draw [line width=0.2mm,color=red,domain=-15.611921156381017:15.620855067047883] plot(\x,{(--1.8--1.3*\x)/1.});
\draw [line width=0.2mm,color=red,domain=-15.611921156381017:15.620855067047883] plot(\x,{(--3.75--0.9285714285714286*\x)/1.});
\begin{scriptsize}
\end{scriptsize}

\end{tikzpicture}
\end{center} 

\caption{\label{fig:The-18-lines-flash}The 18 lines of $\protect\cA_{1}$
for the flash arrangement of $\protect\L_{\{2\},\{6\}}$}
\end{figure}

\subsection{\label{subsec:A-flashing-16-lines}A flashing configuration of $8$
lines for $\protect\L_{\{2\},\{4\}}$}

For a parameter $t$, consider the following configuration of $12$
lines $\ell_{1},\dots,\ell_{12}$ whose normals are the columns of
the following matrix
\[
\left(\begin{array}{cccccccccccc}
1 & 0 & t^{2}-\tfrac{1}{2}t & 2t^{2} & 1 & 0 & t & 2 & 0 & 2t^{2}-t & t & 1\\
0 & 0 & t^{2}-t+\tfrac{1}{2} & 2t^{2}-2t+1 & 1 & 1 & t-1 & 2t-1 & 1 & 2t^{2}-3t+1 & t-\tfrac{1}{2} & 1\\
0 & 1 & t^{2}-\tfrac{1}{2}t & t^{2} & 1 & t & 0 & t & 0 & t^{2}-t & t-\tfrac{1}{2} & t
\end{array}\right).
\]
For $t\notin\cS=\{\pm1,0,\frac{1}{2},2,\tfrac{1}{2}(1\pm i)\}$, the
image by $\Lambda_{\{2\},\{4\}}$ of the union $\co=\co(t)$ of the
first eight lines is the configuration $\cC_{1}=\{\ell_{5},\dots,\ell_{12}\}$,
moreover $\L_{\{2\},\{4\}}(\cC_{1})=\co$. The line arrangements $\co$
and $\cC_{1}$ are projectively equivalent.  

The $12$ lines and $16$ triple points of $\cA_{1}$ form a $(16_{3},12_{4})$-configuration.
The arrangement $\cL_{\{3\}}(\cD(\cA_{1}))$ has $16$ lines with
$t_{2}=48,t_{4}=12$ and form a $(12_{4},16_{3})$-configuration. 
\begin{rem}
The Reye configuration is a well-known $(12_{4},16_{3})$-configuration
(see e.g. \cite{Dolgachev}), however we checked that the matroids
associated to the Reye configuration and to our $(12_{4},16_{3})$-configuration
are not isomorphic. 
\end{rem}

For $t=\tfrac{1}{2}(1\pm i)$, one has that $\cC_{1}=\L_{\{2\},\{4\}}(\cC_{0})$
is the $12$ line arrangement $\{\ell_{1},\dots,\ell_{12}\}$, and
$\L_{\{2\},\{4\}}(\cC_{1})=\emptyset$. Remarkably, for that value
of $t,$ the line arrangement $\cC_{1}$ is the Ceva(4) line arrangement.

The dual arrangements $\check{\cC_{k}}=\cD_{\{2\}}(\check{\cC_{k}}),\,k=1,2$
have $22$ lines moreover $\L_{\{4\},\{2\}}(\check{\cC_{0}})=\check{\cC_{1}},$
$\L_{\{4\},\{2\}}(\check{\cC_{1}})=\check{\cC_{0}}$; their singularities
are $t_{2}=99,t_{4}=8,t_{7}=4.$ The arrangements $\check{\cC_{0}},\check{\cC_{1}}$
are not free. The union $\check{\cC_{0}}\cup\check{\cC_{1}}$ has
$28$ lines and singularities $t_{2}=180,\,t_{4}=5,\,t_{7}=8.$

\section{\label{sec:Matroids-and-line}Matroids and line operators}

A matroid is a pair $(E,\cB)$, where $E$ is a finite set and the
elements of $\cB$ are subsets of $E$ (called \textit{basis),} subject
to the following properties:\\
$\bullet$ $\cB$ is non-empty\\
$\bullet$ if $A\in\cB$ and $B\in\cB$ are distinct basis and $a\in A\setminus B$,
then there exists $b\in B\setminus A$ such that $(A\setminus\{a\})\cup\{b\}\in\cB$. 

The basis have the same order $n$, called the \textit{rank} of $(E,\cB)$.
Order $n$ subsets of $E$ that are not basis are called \textit{non-basis}.
We identify $E$ with $\{1,\dots,m\}$. A realization (over some field)
of the matroid $(E,\cB)$ is a the data of a size $n\times m$ matrix
$M$, with columns $C_{1},\dots,C_{m}$, such that an order $n$ subset
$\{i_{1},\dots,i_{n}\}$ of $E$ is a non-base if and only if the
size $n$ minor $|C_{i_{1}},\dots,C_{i_{n}}|$ is zero. 

Since we are mainly working on the projective plane, we are interested
by rank $3$ matroid. A realization of a matroid is then a labelled
line arrangement $\cC=\{L_{1},\dots,L_{m}\}$ in $\PP^{2}$ such that
any set of three lines $L_{i},L_{j},L_{k}$ of $\cC$ meet at one
point if and only if $i,j,k$ is a non-basis of $(E,\cB)$ . 

Conversely, any labelled line arrangement $\cC=\{L_{1},\dots,L_{m}\}$
defines a rank $3$ matroid $M(\cC)=(\{1,\dots,m\},\cB)$, where the
non-base are any triple $\{i,j,k\}$ such that the intersection of
$L_{i},L_{j},L_{k}$ is non-empty. 

Given a matroid $(E,\cB)$, one may compute the moduli space $\cM_{\cB}$
of realizations (see \cite{RoulleauDy} for an example of such computation). 

Examples in Section \ref{subsec:Zariski-pair-arrangements} shows
that the moduli space $\cM_{\cB}$ may have several connected components,
and the action of some line operators $\L$ on some elements of these
components may be different: the resulting line arrangements do not
define the same matroids. These examples even show that inside the
same irreducible components, two distinct realization may have images
by $\L$ defining distinct matroids. 

The example of unassuming arrangements shows that even for the trivial
matroid (no non-bases) over $E=\{1,\dots,6\}$, the line operators
may have different behaviors.

\vspace{3mm}

\noindent Xavier Roulleau\\
Université d'Angers, \\
CNRS, LAREMA, SFR MATHSTIC, \\
F-49000 Angers, France 

\noindent xavier.roulleau@univ-angers.fr
\end{document}